\documentclass[reqno]{amsart}
\usepackage{latexsym,amsmath,amssymb,amscd}
\usepackage[all]{xy}
\usepackage{enumerate}
\usepackage{hyperref} 
\makeatletter
\@addtoreset{figure}{section}
\def\thefigure{\thesection.\@arabic\c@figure}
\def\fps@figure{h,t}
\@addtoreset{table}{bsection}

\def\thetable{\thesection.\@arabic\c@table}
\def\fps@table{h, t}
\@addtoreset{equation}{section}

\makeatother

\makeatletter
\newcommand\@dotsep{4.5}
\def\@tocline#1#2#3#4#5#6#7{\relax
	\ifnum #1>\c@tocdepth % then omit
	\else
	\par \addpenalty\@secpenalty\addvspace{#2}%
	\begingroup \hyphenpenalty\@M
	\@ifempty{#4}{%
		\@tempdima\csname r@tocindent\number#1\endcsname\relax
	}{%
		\@tempdima#4\relax
	}%
	\parindent\z@ \leftskip#3\relax \advance\leftskip\@tempdima\relax
	\rightskip\@pnumwidth plus1em \parfillskip-\@pnumwidth
	#5\leavevmode\hskip-\@tempdima #6\relax
	\leaders\hbox{$\m@th
		\mkern \@dotsep mu\hbox{.}\mkern \@dotsep mu$}\hfill
	\hbox to\@pnumwidth{\@tocpagenum{#7}}\par
	\nobreak
	\endgroup
	\fi}
\makeatother 
\newtheorem{theorem}{Theorem}
\newtheorem{corollary}[theorem]{Corollary}
\newtheorem{definition}[theorem]{Definition}
\newtheorem{example}[theorem]{Example}
\newtheorem{lemma}[theorem]{Lemma}

\newtheorem{proposition}[theorem]{Proposition}
\newtheorem{remark}[theorem]{Remark}

\numberwithin{theorem}{section}
\numberwithin{equation}{section}

\newcommand{\0}{{\bf 0}}
\renewcommand{\1}{{\bf 1}}

\newcommand{\ad}{{\rm ad}}

\newcommand{\Aut}{{\rm Aut}}
\newcommand{\Der}{{\rm Der}}

\newcommand{\cone}{{\mathbf C}}
\newcommand{\de}{{\rm d}}
\newcommand{\ee}{{\rm e}}
\newcommand{\End}{{\rm End}}
\newcommand{\ev}{{\rm ev}}

\newcommand{\GL}{{\rm GL}}
\newcommand{\Hom}{{\rm Hom}}
\newcommand{\id}{{\rm id}}
\newcommand{\ie}{{\rm i}}

\newcommand{\Ker}{{\rm Ker}\,}

\newcommand{\opn}{\operatorname}
\newcommand{\Prim}{{\rm Prim}}

\renewcommand{\Re}{{\rm Re}}
\newcommand{\spa}{{\rm span}\,}
\newcommand{\spec}{{\rm spec}}

\newcommand{\susp}{{\mathbf S}}

\newcommand{\Tr}{{\rm Tr}\,}

\newcommand{\CC}{{\mathbb C}}

\newcommand{\KK}{{\mathbb K}}
\newcommand{\NN}{{\mathbb N}}

\newcommand{\RR}{{\mathbb R}}
\newcommand{\TT}{{\mathbb T}}
\newcommand{\ZZ}{{\mathbb Z}}

\newcommand{\Ac}{{\mathcal A}}
\newcommand{\Bc}{{\mathcal B}}
\newcommand{\Cc}{{\mathcal C}}
\newcommand{\Dc}{{\mathcal D}}
\newcommand{\Gc}{{\mathcal G}}
\newcommand{\Hc}{{\mathcal H}}
\newcommand{\Ic}{{\mathcal I}}
\newcommand{\Jc}{{\mathcal J}}
\newcommand{\Kc}{{\mathcal K}}

\newcommand{\Oc}{{\mathcal O}}
\newcommand{\Pc}{{\mathcal P}}

\newcommand{\Tc}{{\mathcal T}}
\newcommand{\Uc}{{\mathcal U}}
\newcommand{\Vc}{{\mathcal V}}
\newcommand{\Xc}{{\mathcal X}}
\newcommand{\Yc}{{\mathcal Y}}

\newcommand{\Pg}{{\mathfrak P}}

\renewcommand{\gg}{{\mathfrak g}}
\newcommand{\hg}{{\mathfrak h}}

\renewcommand{\ng}{{\mathfrak n}}

\newcommand{\zg}{{\mathfrak z}}
\newcommand{\wtilde}[1]{\widetilde{#1}}
\newcommand{\what}[1]{\widehat{#1}}
\newcommand{\matt}[2]%{\ensuremath{{#1}\oplus\cdots\oplus{#2}}}
{\ensuremath{\begin{pmatrix}
			{#1} & 0 \\
			   0 & {#2}
		\end{pmatrix}}}

\newcommand{\matto}[2]{\ensuremath{{#1}\oplus{#2}}}

\newcommand{\mattt}[3]%{\ensuremath{{#1}\oplus\cdots\oplus{#2}}}
{\ensuremath{\begin{pmatrix}
			{#1} & & 0\\
			& {#2} & \\
			0& & {#3}
\end{pmatrix}}}

\title[Stably finiteness of Lie group $C^*$-algebras]{On stably finiteness for $C^*$-algebras of exponential solvable Lie groups}

\author{Ingrid Belti\c t\u a}
\author{Daniel Belti\c t\u a}
\address{Institute of Mathematics ``Simion Stoilow'' of the Romanian Academy,
P.O. Box 1-764, Bucharest, Romania}

\email{Ingrid.Beltita@imar.ro, ingrid.beltita@gmail.com}
\email{Daniel.Beltita@imar.ro, beltita@gmail.com}

\begin{document}

\begin{abstract} 
We study the link between stably finiteness and stably pro\-jec\-tion\-less-ness  for $C^*$-algebras of solvable Lie groups. 
We show that these two properties are equivalent if the dimension of the group is not divisible by $4$; otherwise, they are not necessarily equivalent. 
To provide examples proving the last assertion, we study
exponential solvable Lie groups that have 
nonempty finite open sets in their unitary dual. 
\end{abstract}

\maketitle

%\tableofcontents

\section{Introduction}

Finite approximation properties of Lie group $C^*$-algebras, 
particularly quasidiagonality and AF-embeddability, have recently been investigated in \cite{BB18b} and \cite{BB21}, relating them to 
special topological properties of the their corresponding Lie algebras,  and of primitive ideal spaces. 
The study of the stably finiteness for larger classes of $C^*$-algebras of solvable Lie groups, which we initiate in this paper, requires new techniques that take into account their real structure (cf. \cite{Ros16}), as explained below. 

For every separable exact $C^*$-algebra $\Ac$, the lack of nonempty quasi-compact open subsets of the primitive ideal 
space $\Prim(\Ac)$ ensures that $\Ac$ is AF-embeddable, that is, $\Ac$ is isomorphic to a closed subalgebra of an 
AF-algebra. 
The converse also holds if $\Ac$ is traceless (cf. \cite[Cor.~B-C]{Ga20}), and for other $C^*$-algebras as well. 
For instance, if  $G$ is a generalized $ax+b$-group, 
%hence a solvable Lie group $G = \Vc \rtimes \RR$, where 
%$\Vc$ is a finite-dimensional real vector space, 
then $C^*(G)$ is not traceless.  
Nevertheless, $C^*(G)$ is AF-embeddable if and only if $\Prim(G)$ has no nonempty  quasi-compact open  subsets. 
(See \cite{GKT92} and Example~\ref{ax+b} below.)  
Here and throughout this paper, we denote $\Prim(G):=\Prim(C^*(G))$ for any locally compact group~$G$, and a topological space is called quasi-compact if all its open covers have finite subcovers, without requiring any separation property.

The questions addressed in the present paper concern the relation between  existence of nonempty quasi-compact open subsets of $\Prim(G)$, stably finiteness or even AF-embeddability of $C^*(G)$ for solvable Lie groups $G$, 
and properties of their corresponding Lie algebras.

The new examples we study here are exponential solvable Lie groups, 
and our most detailed results apply to that class of groups,  
for two main reasons: 
Firstly, for all simply connected solvable Lie groups  with polynomial growth, we have  proved in \cite{BB21} that there exist no  nonempty  quasi-compact open subsets in the primitive ideal spaces of their $C^*$-algebras. 
Secondly, there exist large classes of exponential solvable Lie groups $G$ for which $\Prim(G)$ has finite (hence quasi-compact) open subsets or, more generally, the stabilized $C^*$-algebra $\Kc \otimes C^*(G)$ contains non-zero projections, that is, 
$C^*(G)$ is not stably projectionless. 
See for instance \cite{GKT92}, \cite{KT96}, or Section~\ref{Sect4} below. 

Our main results for 
solvable Lie groups~$G$ could be summarized as follows: 
\begin{enumerate}[(1)]
\item\label{item_1} If $\dim G\in 4\ZZ + 2$ or $\dim G\in 2 \ZZ +1$ then  $C^*(G)$ is  stably finite if and only if  it  is stably projectionless 
(Theorem~\ref{4n+2}). 
\item\label{item_2} If $\dim G\in 4\ZZ$, we prove by examples of exponential solvable Lie groups that
both situations can appear: 
if $\Prim(G) $ has finite open subsets, then $C^*(G)$ could be either AF-embeddable, 
  or not even stably finite 
 (Proposition~\ref{Heis} and Theorem~\ref{N6N15}). 
 Our examples seem to point to an interesting phenomenon that deserves to be investigated in the future. 
 Specifically, we construct families of Lie algebras $\gg_t$ whose structure constants depend continuously on a parameter $t\in T$, 
 and the set $\{t\in T\mid C^*(G_t)\text{ is stably finite}\}$ 
 is closed in the parameter space~$T$ in each single example we have considered. 
\end{enumerate}
We also take the first steps towards describing the exponential solvable Lie groups $G$ whose nilradical is 1-codimensional and for which $C^*(G)$ is AF-embeddable while $\Prim(G)$ contains finite open subsets (Corollary~\ref{cf-cor8}). 

The methods we use for obtaining some of these results owe much to the deep work \cite{Sp88} on AF-embeddings of $C^*$-algebra extensions. 
In addition, we have used the way the ``real'' structures of group  $C^*$-algebras (in the sense of \cite{Ka80} and \cite{Ros16}) are encoded in the K-theory, 
and also the propagation of stably finiteness of the group $C^*$-algebras through suitable deformations of the structure constants of the Lie algebras.

In more detail, this paper has the following contents. 
In Section~\ref{section1}  we obtain some technical results that involve the link between stably finiteness and existence of projections. 
We also investigate the interaction between ``real'' structures of $C^*$-algebras and Rieffel's construction 
of Connes' Thom isomorphism (Proposition~\ref{tsigns}) with an application to solvable Lie groups (Corollary~\ref{signs_solvable}). 
Then we establish that the continuous deformations preserve stably finiteness for certain continuous fields of $C^*$-algebras (Proposition~\ref{prop-cf2}), and  establish a  condition for the failure of stably finiteness in terms of open points of the primitive ideal space (Proposition~\ref{proj}). 
 
Section~\ref{section4k} begins with our main stably-finiteness result on solvable Lie groups whose dimension is not divisible by~4 (Theorem~\ref{4n+2}). 
We then turn to groups whose nilradical is 1-codimensional. 
These groups are basically determined by a derivation of a nilpotent Lie algebra $D\in\Der(\ng)$ and the main task is to describe the stably finiteness or AF-embeddability properties of the $C^*$-algebra of the corresponding group $N\rtimes_D\RR$ in terms of the spectrum of~$D$. 
In this connection, the technique of continuous fields allows us to establish a necessary condition
for stably finiteness in terms of the spectrum or the involved derivation (Theorem~\ref{prop-cf4}). 
We then obtain a technical result that explains the greater complexity of the behaviour of the groups whose dimension is divisible by~4 (Theorem~\ref{cf-prop7}), 
and then we draw a key consequence that is effective by way of decreasing the dimensions in the study of the specific examples (Corollary~\ref{cf-cor8}). 

Finally, in Section~\ref{Sect4}, we use the techniques of Sections \ref{section1} and \ref{section4k} in order to study the $C^*$-algebras of some specific exponential solvable Lie groups. 
We focus on groups whose nilradical is 1-dimensional and 2-step nilpotent 
since this is the simplest class of groups after the generalized $ax+b$-groups (Example~\ref{ax+b}).  
Our most complete results are obtained in the case when the nilradical is a Heisenberg group (Proposition~\ref{Heis}) or is a central extension of the free 6-dimensional 2-step nilpotent Lie group (Theorem~\ref{N6N15}), 
when we characterize the group $C^*$-algebra properties in terms of the spectral  data of the derivation involved in the construction.
In particular, when the Heisenberg group has dimension $4k+3$, there are situations when 
there exists an open point in the unitary dual of $H_n \rtimes \RR$,  but  its $C^*$-algebra 
is AF-embeddable.
We also discuss a class of Heisenberg-like groups associated to finite-dimensional real division algebras (Theorem~\ref{N6N17}). 
These last examples show, in particular,
that the  necessary condition for stably finiteness is 
not sufficient, that is, the lack of stably finiteness is not preserved by continuous deformations.

\subsection*{General notation}
We denote the Lie groups by upper case Roman letters and their Lie algebras by the corresponding
lower case Gothic letters. 
By a solvable/nilpotent Lie group we always understand
a connected simply connected solvable/nilpotent Lie group. 
An exponential Lie group is a Lie group $G$ whose exponential map
$\exp_G \colon\gg \to G$ is bijective. 
All exponential Lie groups are solvable. 
See for instance \cite{ArCu20} for more details. 
For any Lie algebra $\gg$ with its linear dual space $\gg^*$ we denote by $\langle\cdot,\cdot\rangle\colon\gg^*\times\gg\to\RR$ the corresponding duality pairing. 
The coadjoint isotropy subalgebra at any $\xi\in\gg^*$ is  $\gg(\xi):=\{X\in\gg\mid 
%\langle\xi,[X,\gg]\rangle
(\forall Y\in\gg)\ \langle \xi,[X,Y]\rangle=0\}$.

 \section{$K$-theoretic tools, stably finiteness, and  AF-embeddability}\label{section1}

This section contains technical results that play a key role in the next sections.

\subsection{Notation related to the construction of $K$-groups}
\label{App_K}

We start by reminding notions and notation needed in our paper.
Throughout this paper we use the notation from \cite{RLL00} and \cite{Di64}.

For any $C^*$-algebra $\Ac$ its unitization is $\widetilde{\Ac}:=\CC\1\dotplus A$. 
Also, for any integers $m,n\ge 1$ we denote by $M_{m,n}(\Ac)\subseteq M_{m,n}(\widetilde{\Ac})$ the $m\times n$ matrix spaces with entries in $\Ac$ and $\widetilde{\Ac}$, respectively, 
and for $m=n$ we write as usually $M_n(\Ac)\subseteq M_n(\widetilde{\Ac})$ for the corresponding matrix $C^*$-algebras. 
Moreover, $\1_n\in M_n(\widetilde{\Ac})$ 
the identity matrix and $\0_n\in M_n(\Ac)$ is the zero matrix.

We denote $\Pg(\Ac):=\{p\in A\mid p=p^2=p^*\}$ and $\Pg_n(\Ac):=\Pg(M_n(\Ac))$ for any $n\ge 1$. 
The disjoint union 
$$\Pg_\infty(\Ac):=\bigsqcup_{n\ge 1}\Pg_n(\Ac)$$ 
has the natural structure of a graded (noncommutative) semigroup 
with its operation $\oplus$ defined by 
$$\Pg_n(\Ac)\times\Pg_m(\Ac)\to\Pg_{n+m}(\Ac),\quad 
(p,q)\mapsto\matto{p}{q}:=\matt{p}{q}$$ 
for all $m,n\ge 1$. 
The Cartesian projection $s\colon \widetilde{\Ac}\to\CC\1(\subseteq\widetilde{\Ac})$ is extended to 
$s\colon M_n(\widetilde{A})\to M_n(\widetilde{\Ac})$, 
$(a_{ij})_{i,j}\mapsto (s(a_{ij}))_{i,j}$, for any $n\ge 1$. 
We recall the equivalence relation $\sim_0$ on $\Pg_\infty(\widetilde{\Ac})$ defined in the following way: 
If $p\in\Pg_m(\widetilde{\Ac})$ and $q\in\Pg_n(\widetilde{\Ac})$, 
then $p\sim_0 q$ if and only if there exists $v\in M_{m,n}(\widetilde{\Ac})$ with $v^*v=p$ and $vv^*=q$. 
There is a natural additive map $\Pg_\infty(\widetilde{\Ac}) \to K_0(\widetilde{\Ac})$, $p \mapsto [p]_0$.
Then, with the above notation, we have
$$K_0(\Ac)
=\{[p]_0-[s(p)]_0\mid p\in \Pg_\infty(\widetilde{\Ac})\}
\subseteq 
\{[p]_0-[q]_0\mid p,q\in \Pg_\infty(\widetilde{\Ac})\}
=K_0(\widetilde{\Ac}) $$
We also define 
$K_0(\Ac)^+:=\{[p]_0\mid p\in\Pg_\infty(\Ac)\}\subseteq K_0(\Ac)$ 
and, when we wish to emphasize the $C^*$-algebra $\Ac$, 
we write $[p]_{0,\Ac}:=[p]_0\in K_0(\Ac)$ for $p\in\Pg_\infty(\Ac)$. 

For any $C^*$-algebra $\Bc$ and any $*$-morphism $\varphi\colon \Ac\to \Bc$ there is a group morphism $K_0(\varphi)\colon K_0(\Ac)\to K_0(\Bc)$, 
$[p]_0-[s(p)]_0\mapsto[\widetilde{\varphi}(p)]_0-[s(\widetilde{\varphi}(p))]_0$.

We denote $\Uc(\widetilde{\Ac}):=\{u\in \widetilde{A}\mid u^*u=uu^*=\1\}$ and $\Uc_n(\widetilde{\Ac}):=\Uc(M_n(\widetilde{\Ac}))$ for every $n\ge 1$, which is the basic ingredient in the construction
of the group $K_1(\Ac)=K_1(\widetilde{\Ac})$.

\subsection{A $K$-theoretic condition for stably finiteness}
The next proposition  is a partial generalization of \cite[Lemma 1.5]{Sp88} and \cite[Prop. 5.1.5(iii)]{RLL00}, and
it is one of the main tools in this paper. 

\begin{proposition}\label{P1}
For every $C^*$-algebra $\Ac$ the following assertions are equivalent: 
\begin{enumerate}[{\rm(i)}]
	\item\label{P1_item1} 
	There exists 
%	 $k\ge 1$ and $p\in\Pg_k(\Ac)\setminus\{\0_k\}$ 
$p\in\Pg_\infty(\Ac)\setminus\{0\}$ 
	 with $[p]_0=0\in K_0(\Ac)$. 
	\item\label{P1_item2}  
	There exist $r\ge 1$ and $v\in M_r(\widetilde{\Ac})$ with $vv^*=\1_r\ne v^*v$.
	\item\label{P1_item3} The $C^*$-algebra $\Ac$ is not stably finite.   
\end{enumerate}
\end{proposition}

\begin{proof}
\eqref{P1_item1}$\Rightarrow$\eqref{P1_item2}: 
There is $k\ge 1$ with $p\in\Pg_k(\Ac)$, hence $s(p)=\0_k$. 
Then, by \cite[4.2.2(ii)]{RLL00}, there exists $m\ge 1$ with 
$\matt{p}{\1_m}\sim\matt{\0_k}{\1_m}$ in $\Pg_{k+m}(\widetilde{\Ac})$. 
Furthermore, by \cite[2.2.8(i)]{RLL00}, we obtain 
$\mattt{p}{\1_m}{\0_{k+m}}\sim_u\mattt{\0_k}{\1_m}{\0_{k+m}}$ in $\Pg_{2(k+m)}(\widetilde{\Ac})$. 
That is, there exists $w\in\Uc_{2(k+m)}(\widetilde{\Ac})$ with 
$$\mattt{p}{\1_m}{\0_{k+m}}= w\mattt{\0_k}{\1_m}{\0_{k+m}}w^*
\in \Pg_{2(k+m)}(\widetilde{\Ac}).$$
Now, defining 
$$v:=w\mattt{\0_k}{\1_m}{\0_{k+m}}+
\mattt{\1_k-p}{\0_m}{\1_{k+m}}\in M_{2(k+m)}(\widetilde{\Ac})$$
we obtain 
\begin{align*}
\allowdisplaybreaks
vv^*&=w\mattt{\0_k}{\1_m}{\0_{k+m}}w^* 
+\mattt{\1_k-p}{\0_m}{\1_{k+m}} 
=\mattt{\1_k}{\1_m}{\1_{k+m}} \\
&=\1_{2(k+m)}
\end{align*}
and 
\allowdisplaybreaks
\begin{align*}
v^*v
=
&\mattt{\0_k}{\1_m}{\0_{k+m}}w^* w\mattt{\0_k}{\1_m}{\0_{k+m}} \\
&+ \mattt{\0_k}{\1_m}{\0_{k+m}}w^*\mattt{\1_k-p}{\0_m}{\1_{k+m}} \\
&+ \mattt{\1_k-p}{\0_m}{\1_{k+m}}w\mattt{\0_k}{\1_m}{\0_{k+m}} 
+  \mattt{\1_k-p}{\0_m}{\1_{k+m}} \\
=
&
\mattt{\0_k}{\1_m}{\0_{k+m}}+ \0_{2(k+m)}+\0_{2(k+m)}
+ \mattt{\1_k-p}{\0_m}{\1_{k+m}} \\
=
&
\mattt{\1_k-p}{\1_m}{\1_{k+m}} \\
\ne 
&\1_{2(k+m)}.
\end{align*}

\eqref{P1_item2}$\Rightarrow$\eqref{P1_item1}: 
For $r\ge 1$ and $v\in M_r(\widetilde{\Ac})$ with $vv^*=\1_r\ne v^*v$ 
we define $p:=\1_r-v^*v\in\Pg_r(\widetilde{\Ac})\setminus\{\0_r\}$. 
Since $vv^*=\1_r$, we obtain $\1_r=s(vv^*)=s(v)s(v)^*$ in $M_r(\CC\1)$, 
hence $s(v)^*s(v)=\1_r$. 
This implies $s(p)=\1_r-s(v^*v)=\1_r-s(v)^*s(v)=\0_r$ 
It follows that $p \in \Pg_r(\Ac)$.  
Moreover, $p+v^*v=vv^*$ and $p v^*v=\0_r$ hence, by \cite[3.1.7(iv)]{RLL00}, 
$[p]_0+[v^*v]_0=[vv^*]_0$ in $K_0(\widetilde{A})$. 
Here $v^*v\sim_0 vv^*$, hence $[v^*v]_0=[vv^*]_0$ in $K_0(\widetilde{A})$, and we then obtain $[p]_0=0\in K_0(\Ac)$.

\eqref{P1_item2}$\Rightarrow$\eqref{P1_item3}: 
This is just the definition of stably finite $C^*$-algebras. 
\end{proof}

We now obtain the following slight improvement of \cite[Lemma 1.5]{Sp88}. 

\begin{corollary}\label{Sp88_Lemma1.5}
If $\Ac$ is a $C^*$-algebra with a closed two-sided ideal $\Jc$ 
and the corresponding inclusion map $\varphi\colon\Jc\hookrightarrow\Ac$, 
then the following assertions hold: 
\begin{enumerate}[{\rm (i)}]
	\item\label{Sp88_Lemma1.5_item1} 
	If $\Ac$ is stably finite then $K_0(\Jc)^+\cap \Ker K_0(\varphi)=\{0\}$ 
	and $\Jc$ is stably finite. 
	\item\label{Sp88_Lemma1.5_item2}
	If  $K_0(\Jc)^+\cap \Ker K_0(\varphi)=\{0\}$ and both $\Jc$ and $\Ac/\Jc$ are stably finite, then $\Ac$ is stably finite. 
	\item\label{Sp88_Lemma1.5_item3} 
	If  $0\to \Jc\hookrightarrow\Ac\to\Ac/\Jc\to0$ is a split exact sequence and both $\Jc$ and $\Ac/\Jc$ are stably finite, then $\Ac$ is stably finite. 
\end{enumerate}
\end{corollary}

\begin{proof}
\eqref{Sp88_Lemma1.5_item1} 
If $x\in K_0(\Jc)^+\cap \Ker K_0(\varphi)$ then there exist $k\ge 1$ and $p\in\Pg_k(\Jc)$ with $x=[p]_{0,\Jc}$ and $0=(K_0(\varphi))(x)=(K_0(\varphi))([p]_{0,\Jc})=[p]_{0,\Ac}\in K_0(\Ac)$. 
Since  $\Ac$ is stably finite, it then follows by Proposition~\ref{P1} that $p=\0_k\in M_k(\Ac)$, hence $x=0\in K_0(\Jc)$. 
In addition, stably finiteness of $\Ac$ directly implies that 
its subalgebra~$\Jc$ is stably finite.

\eqref{Sp88_Lemma1.5_item2} 
This follows directly from \cite[Lemma 1.5]{Sp88},  
but we give here a proof using Proposition~\ref{P1}. 
Let $\psi\colon\Ac\to\Ac/\Jc$ be the quotient map. 
If $p\in\Pg_\infty(\Ac)$ with $[p]_{0,\Ac}=0$, then $0=(K_0(\psi))([p]_{0,\Ac})=[\psi(p)]_{0,\Ac/\Jc}$. 
Since $\Ac/\Jc$ is stably finite, it then follows by Proposition~\ref{P1} that $\psi(p)=0\in\Pg_\infty(\Ac/\Jc)$, that is, $p\in \Pg_\infty(\Jc)$. 
Then, denoting $x:=[p]_{0,\Jc}\in K_0(\Jc)^+$, we have $(K_0(\varphi))(x)=[p]_{0,\Ac}=0$, hence $x\in K_0(\Jc)^+\cap \Ker K_0(\varphi)=\{0\}$. 
Now, since $\Jc$ is stably finite, we obtain $p=\0_k\in M_k(\Jc)$ by Proposition~\ref{P1} again, hence $p=\0_k\in M_k(\Ac)$. 
Therefore a new application of Proposition~\ref{P1} shows that $\Ac$ is stably finite. 

\eqref{Sp88_Lemma1.5_item3} 
We have $\Ker K_0(\varphi)=\{0\}$ by the splitness hypothesis, 
hence Assertion~\eqref{Sp88_Lemma1.5_item2} applies. 
\end{proof}

The following simple facts were already noted in \cite[proof of Cor. D]{Ga20},
 but we prove them here for completeness. 

\begin{corollary}\label{rem-1.5} Let $\Ac$ be a $C^*$-algebra. 
\begin{enumerate}[{\rm (i)}]
\item \label{rem-1.5_i} If $\Ac$ is  stably projectionless, then it is stably finite.
Moreover, a $C^*$-algebra~$\Ac$ with $K_0(\Ac)=\{0\}$ is stably finite if and only if it is stably projectionless.
\item\label{rem-1.5_ii}  If $\Prim(\Ac)$ has no nonempty quasi-compact open subset, 
then $\Ac$ is stably projectionless. 
\item\label{rem-1.5_iii}  If  $\Prim(\Ac)$ has no nonempty quasi-compact open subsets, 
then $\Ac$ is stably finite. 
\end{enumerate}
\end{corollary}

\begin{proof}
Assertion \eqref{rem-1.5_i} is a direct consequence of Proposition~\ref{P1} (\eqref{P1_item1}$\iff$\eqref{P1_item3}). 

\eqref{rem-1.5_ii}
If $\Ac$ is not stably projectionless, 
there exist 
$k\ge 1$ and $p\in\Pg_k(\Ac)\setminus\{0\}$, hence
there exists a nonempty quasi-compact open subset of $\widehat{M_k(\Ac)}$. 
(See for instance the proof of \cite[Ex. 4.8((vii)$\Rightarrow$(iv))]{BB21}.) 
Since $M_k(\Ac)=M_k(\CC)\otimes \Ac$, 
it follows that $\widehat{M_k(\Ac)}$ is homeomorphic to $\widehat{\Ac}$, 
by \cite[Th. B.45(b)]{RaWi98}. 
Therefore $\widehat{\Ac}$  has a nonempty quasi-compact open subset. 
Moreover, the canonical mapping $\widehat{\Ac}\to\Prim(\Ac)$, $[\pi]\mapsto\Ker\pi$, 
is continuous and open, hence it maps any nonempty quasi-compact 
open subset of $\widehat{\Ac}$ onto a nonempty quasi-compact open subset of $\Prim(\Ac)$.
It follows that  $\Prim(\Ac)$ has a nonempty quasi-compact open subset, 
which is a contradiction with the hypothesis. 

Assertion~\eqref{rem-1.5_iii}  follows immediately from  \eqref{rem-1.5_i}  and \eqref{rem-1.5_ii}. 
\end{proof}

\begin{remark}
\normalfont
In the special case of separable exact $C^*$-algebras, 
Corollary~\ref{rem-1.5} \eqref{rem-1.5_iii} is a weak version of \cite[Cor. B]{Ga20}, which gives AF-embeddability rather than just stable finiteness. 
\end{remark}

\subsection{Action of ``real'' structures on $K$-groups} 

The following terminology goes back to G.G. Kasparov \cite{Ka80}.
\begin{definition}\label{real} 
\normalfont
A \emph{``real'' structure} of a $C^*$-algebra $\Ac$  is an antilinear mapping $\tau\colon \Ac\to \Ac$,  satisfying $\tau(ab)=\tau(a)\tau(b)$, $\tau(a^*)=\tau(a)^*$, 
and $\tau(\tau(a))=a$ for all $a,b\in\Ac$. 
A \emph{``real'' $C^*$-algebra} is a $C^*$-algebra $\Ac$ with a fixed ``real'' structure~$\tau$.
We  denote $\overline{a}:=\tau(a)$ for all $a\in \Ac$ when no confusion arise.

A \emph{``real'' ideal} of $\Ac$ is a closed two-sided ideal $\Jc\subseteq \Ac$ 
that is invariant to the ``real'' structure of $\Ac$. 
In this case $\Jc$ is a ``real'' $C^*$-algebra with respect to the ``real'' structure $\tau\vert_\Jc\colon \Jc\to \Jc$. 

Let $\Ac$, $\Bc$ be $C^*$-algebras with ``real'' structures $\tau_\Ac$ and $\tau_\Bc$, respectvely. 
A  \emph{``real'' morphism} is a $*$-morphism $\psi\colon \Ac\to \Bc$ satisyfing $\psi(\tau_\Ac (a))=\tau_\Bc(\psi(a))$ for all $a\in \Ac$. 
\end{definition}

For every $n\ge 1$ the matrix algebra $M_n(\CC)$ has a canonical ``real'' structure;
if 
$\Ac$ is a ``real'' $C^*$-algebra then the $C^*$-algebra 
$M_n(\Ac)=M_n(\CC)\otimes \Ac$ has a canonical ``real'' structure $(a_{ij})\mapsto (\overline{a_{ij}})$. 
Moreover $\widetilde{\Ac}$ has a canonical ``real'' structure given by $\overline{a+ z\1}=\overline{a}+\overline{z}\1$ for every $a\in \Ac$ and $z\in\CC$. 

If $u,v\in\Uc_\infty(\widetilde{\Ac})$ then $\overline{u},\overline{v}\in \Uc_\infty(\widetilde{\Ac})$ and $\overline{\matto{u}{v}}=\matto{\overline{u}}{\overline{v}}\in\Uc_\infty(\widetilde{\Ac})$, and 
$$u\sim_1 v\iff \overline{u}\sim_1\overline{v}.$$
(See \cite[8.1.1]{RLL00}.) 
Therefore we obtain a well-defined group homomorphism 
$$K_1(\widetilde{\Ac})\to K_1(\widetilde{\Ac}),\quad [u]_1\mapsto \overline{[u]_1}:=[\overline{u}]_1$$
which is actually an isomorphism and is equal to its own inverse. 

If $p,q\in\Pg_\infty(\widetilde{\Ac})$ then $\overline{p},\overline{q}\in\Pg_\infty(\widetilde{\Ac})$ 
and 
$\overline{\matto{p}{q}}=\matto{\overline{p}}{\overline{q}}\in\Uc_\infty(\widetilde{\Ac})$, 
\begin{align*}
p\sim_0 q\iff \overline{p}\sim_0\overline{q}
\end{align*}
and $s(\overline{p})=\overline{s(p)}$. 
We then obtain a well-defined semigroup homomorphism 
$$\Dc(\widetilde{\Ac})\to \Dc(\widetilde{\Ac}), \quad [p]_\Dc\mapsto\overline{[p]_\Dc}:=[\overline{p}]_\Dc$$
which is actually an isomorphism and is equal to its own inverse. 
This further gives rise to a group isomorphism 
$$K_0(\widetilde{\Ac})=G(\Dc(\widetilde{\Ac}))\to G(\Dc(\widetilde{\Ac}))=K_0(\widetilde{\Ac}), \quad [p]_0\mapsto\overline{[p]_0}:=[\overline{p}]_0$$
which  is equal to its own inverse, and satisfies 
$$[s(\overline{p})]_0=\overline{[s(p)]_0}\text{ for all }p\in\Pg_\infty(\widetilde{\Ac}).$$
If $\psi\colon \Ac\to \Bc$ is a ``real'' morphism of ``real'' $C^*$-algebras, 
then its corresponding group morphism 
$K_j(\widetilde{\psi})\colon K_j(\widetilde{\Ac})\to K_j(\widetilde{\Bc})$ 
satisfies $K_j(\widetilde{\psi})(\overline{x})=\overline{K_j(\widetilde{\psi})(x)}$ for all $x\in K_j(\widetilde{\Ac})$ and $j=0,1$. 

In particular, for $j=0$, $\Bc=\{0\}$, and $\psi=0$, it follows that the subgroup  $K_0(\Ac)=\Ker(K_0(\widetilde{\psi}))$ is invariant to the automorphism $x\mapsto\overline{x}$ of $K_0(\widetilde{\Ac})$. 

\begin{lemma}\label{deltas}
Let $\psi\colon \Ac \to \Bc$ be  a ``real'' surjective morphism of ``real'' $C^*$-algebras.  
Denote $\Jc:=\Ker\psi$, regarded as a ``real'' ideal of $\Ac$ 
with its corresponding inclusion map $\varphi\colon \Jc\hookrightarrow \Ac$, 
and consider the six-term exact sequence 
\begin{equation}\label{hexagon}
\xymatrix{
K_0(\Jc) \ar[r]^{K_0(\varphi)} & K_0(\Ac) \ar[r]^{K_0(\psi)} & K_0(\Bc) \ar[d]^{\delta_0}\\ 
K_1(\Bc) \ar[u]^{\delta_1} & K_1(\Ac) \ar[l]_{K_1(\psi)} & K_1(\Jc) \ar[l]_{K_1(\varphi)}
}
\end{equation}
Then we have 
\begin{enumerate}[{\rm(i)}]
\item\label{deltas_item1} 
$\delta_0(\overline{x})=-\overline{\delta_0(x)}$ for all $x\in K_0(\Bc)$; 
\item\label{deltas_item2} 
$\delta_1(\overline{y})=\overline{\delta_1(y)}$ for all $y\in K_1(\Bc)$.\end{enumerate}
\end{lemma}

\begin{proof}
\eqref{deltas_item1}
For arbitrary $x\in K_0(\Bc)$ there exist $n\ge 1$ and  $p\in \Pg_n(\widetilde{\Bc})$ with $x=[p]_0-[s(p)]_0$. 
There also exist $a=a^*\in M_n(\widetilde{\Ac})$ and $u\in \Uc_n(\widetilde{\Jc})$ 
with 
$$\widetilde{\psi}(a)=p\; \text{ and }\; \widetilde{\varphi}(u)=\exp(2\pi\ie a)\in\Uc_n(\widetilde{A}),$$ 
and $\delta_0(x)=-[u]_1$ by \cite[12.2.2(i)]{RLL00}. 
Furthermore $\overline{x}=[\overline{p}]_0-[\overline{s(p)}]_0$ 
and $\overline{a}=\overline{a}^*\in M_n(\widetilde{A})$ satisfies $\widetilde{\psi}(\overline{a})=\overline{\widetilde{\psi}(a)}=\overline{p}$. 
On the other hand, 
$\widetilde{\varphi}(\overline{u}^*)
=\overline{\widetilde{\varphi}(u)}^*
=\exp(2\pi\ie \overline{a})$, 
hence, since $[u^*]_1=-[u]_1$ by \cite[8.1.3]{RLL00}, we obtain
$$\delta_0(\overline{x})=-[\overline{u}^*]_1=[\overline{u}]_1=\overline{[u]_1}
=-\overline{\delta_0(x)}.$$
\eqref{deltas_item2} 
For arbitrary $y\in K_1(\Bc)=K_1(\widetilde{\Bc})$ 
there exist $n\ge 1$ and  $u\in \Uc_n(\widetilde{\Bc})$ with $y=[u]_1$. 
There also exist $v\in \Uc_{2n}(\widetilde{\Ac})$ and 
$p\in\Pg_{2n}(\widetilde{\Jc})$ 
with 
$$\widetilde{\psi}(v)=\matt{u}{u^*}\; \text{ and }\;
\widetilde{\varphi}(p)=v\matt{\1_n}{\0_n}v^*, $$ 
and 
we have $\delta_1(y)=[p]_0-[s(p)]_0$ by \cite[9.1.4]{RLL00}. 
Furthermore $\overline{y}=[\overline{u}]_1$ 
and $\overline{u}\in \Uc_n(\widetilde{B})$, 
$\overline{v}\in \Uc_{2n}(\widetilde{A})$, 
$\overline{p}\in\Pg_{2n}(\widetilde{J})$ satisfy 
$$\widetilde{\psi}(\overline{v})
=\overline{\widetilde{\psi}(p)}
=\matt{\overline{u}}{\overline{u}^*}\; \text{ and }\; 
\widetilde{\varphi}(\overline{p})
=\overline{\widetilde{\varphi}(p)}
=\overline{v}\matt{\1_n}{\0_n}\overline{v}^*$$ 
hence 
$$\delta_1(\overline{y})
=[\overline{p}]_0-[s(\overline{p})]_0
=[\overline{p}]_0-[\overline{s(p)}]_0
=\overline{\delta_1(y)}.$$
This completes the proof.  
\end{proof}

\begin{remark}
\normalfont
For any locally compact group $G$ we regard its $C^*$-algebra $C^*(G)$ as a ``real'' $C^*$-algebra with its canonical ``real'' structure given by $\overline{f}(x):=\overline{f(x)}$ for every $f\in\Cc_c(G)\hookrightarrow C^*(G)$. 
See for instance \cite[\S 3.2]{Ros16}. 
\end{remark}

\begin{lemma}\label{dim1}
We have $[\overline{u}]_1=[u]_1\in K_1(C^*(\RR))\simeq\ZZ$ 
for all $u\in\Uc_\infty(C^*(\RR)^\sim)$. 
\end{lemma}

\begin{proof}
We use the well-known $*$-isomorphism given by the Fourier transform 
$$F\colon C^*(\RR)\to\Cc_0(\ie\RR),\quad 
(F(f))(\ie\xi)=\int_{\RR}\ee^{-\ie\xi x}f(x)\de x 
\text{ if }f\in\Cc_c(\RR),$$
where we regard $\Cc_0(\ie\RR)$  as a commutative $C^*$-algebra with its pointwise operations and with the sup-norm, and ``real'' structure $\tau_0\colon \Cc_0(\ie\RR)\to\Cc_0(\ie\RR)$, 
$(\tau_0(g))(z)=\overline{g(\overline{z})}$ for all $g\in\Cc_0(\ie\RR)$ and $z\in\ie\RR$. 
We have 
$$F(\overline{f})(z)=\overline{F(f)(\overline{z})}\text{ for all }z\in\ie\RR\text{ and }f\in\Cc(\RR)$$
hence $F$ is a ``real'' isomorphism of ``real'' $C^*$-algebras. 
Consider the Cayley homeomorphism 
$$\kappa\colon \ie\RR\to\TT\setminus\{-1\},\quad \kappa(\ie\xi)=\frac{\ie\xi+1}{-\ie\xi+1}$$
with its inverse 
$\kappa^{-1}(w)=\frac{w-1}{w+1}$ for all $w\in\TT\setminus\{-1\}$; then
$$\overline{\kappa(z)}=\kappa(\overline{z})\text{ for all }z\in\ie\RR.$$ 
Therefore, the Cayley transform gives a ``real'' isomorphism from the unitization of the ``real'' $C^*$-algebra $\Cc_0(\ie\RR)$ 
onto $\Cc(\TT)$, when $\Cc(\TT)$ is endowed with the ``real'' structure 
$(\tau(h))(w)=\overline{h(\overline{w})}$ for all $h\in\Cc(\TT)$ and $w\in\TT$. 

For the above reasons it suffices to prove that the action of the ``real'' structure~$\tau$ on $K_1(\Cc(\TT))$ is the identity map. 
To this end, we recall that, if we denote $u:=\id_{\TT}\in\Cc(\TT,\TT)=\Uc_1(\Cc(\TT))$, then the mapping 
$$\ZZ\to K_1(\Cc(\TT)),\quad m\mapsto m[u]_1$$
is a group isomorphism, hence for every $ y= [w]_1 \in  K_1(\Cc(\TT))$ there is a unique $m \in \ZZ$ such that 
$y = m [u]_1$. 
On the other hand, it is clear that $\tau(u)=u$, hence $[\tau(w)]_1=m [u]_1\in K_1(\Cc(\TT))$, that is, $\overline{y}=y$ in $K_1(\Cc(\TT))$.
This shows that the action of $\tau$ on~$K_1(\Cc(\TT))$ is the identity map. 
\end{proof}

\begin{definition}\label{rdym}
\normalfont
A \emph{``real'' $C^*$-dynamical system} is a $C^*$-dynamical system $(\Ac,T,\alpha)$, where $\Ac$ is a ``real'' $C^*$-algebra, $T$ is a locally compact group, and $\alpha\colon T\to\Aut \Ac$, $t\mapsto\alpha_t$, satisfies $\alpha_t(\overline{a})=\overline{\alpha_t(a)}$ for all $t\in T$ and $a\in \Ac$. 
 \end{definition}

\begin{lemma}\label{rcrossed}
For every ``real'' $C^*$-dynamical system  $(\Ac,T,\alpha)$ 
its corresponding crossed product $\Ac\rtimes_\alpha T$ has a unique ``real'' structure satisfying $\overline{f}(t):=\overline{f(t)}$ for all $t\in T$ and $f \in\Cc_c(T,\Ac)$. 
\end{lemma}

\begin{proof}
Uniqueness follows from the fact that $\Cc_c(T,\Ac)$ is dense in $\Ac\rtimes_\alpha T$. 

To prove the existence, we first note that 
the antilinear mapping $f\mapsto \overline{f}$ defined as in the statement on $\Cc_c(T,\Ac)$ preserves the multiplication and the involution. 
In fact, we recall that 
$$(f\ast g)(t)=\int_T f(r)\alpha_r(g(r^{-1}t))\de r\text{ and }
f^*(t):=\Delta(t^{-1})\alpha_t(f(t^{-1})^*)$$
hence $\overline{f\ast g}=\overline{f}\ast \overline{g}$ and $\overline{f^*}=\overline{f}^*$ for all $f,g\in\Cc_c(T,\Ac)$. 

It remains to show that $f\mapsto \overline{f}$ is isometric with respect to the $C^*$-norm on $\Cc_c(T,\Ac)$. 
To this end let  $(\pi,U)$ be a covariant representation of $(\Ac,T,\alpha)$ on a complex Hilbert space~$\Hc$, that is, $\pi(\alpha_t(a))=U_t \pi(a)U_t^*\in\Bc(\Hc)$ for all $t\in T$ and $a\in \Ac$. 
For any fixed antilinear involutive isometry $C\colon \Hc\to\Hc$ 
we define $\overline{\pi}\colon \Ac\to\Bc(\Hc)$, $\overline{\pi}(a):=C\pi(\overline{a})C$, and $\overline{U}\colon T\to\Bc(\Hc)$, $\overline{U}_t:=CU_tC$. 
Then it is straightforward to check that $(\overline{\pi},\overline{U})$ is again a  
covariant representation of $(\Ac,T,\alpha)$ on the complex Hilbert space~$\Hc$, and moreover for every $f\in\Cc_c(T,\Ac)$ we have 
\begin{align*}
(\pi\rtimes U)(\overline{f})
&=\int_T\pi(\overline{f}(t))U_t\de t
=\int_T\pi(\overline{f(t)})U_t\de t
=\int_TC\overline{\pi}(f(t))CU_t\de t \\
&=C\Bigl(\int_T\overline{\pi}(f(t))\overline{U}_t\de t \Bigr)C
=C\Bigl((\overline{\pi}\rtimes \overline{U})(f)\Bigr)C.
\end{align*}
This shows that for every covariant representation $(\pi,U)$ there exists a covariant representation  $(\overline{\pi},\overline{U})$  with 
$\Vert (\pi\rtimes U)(\overline{f})\Vert=\Vert (\overline{\pi}\rtimes \overline{U})(f)\Vert$. 
Then, by the definition of the $C^*$-norm on $\Cc_c(T,\Ac)$, we obtain $\Vert f\Vert=\Vert \overline{f}\Vert$ in $\Ac\rtimes_\alpha T$.
This finishes the proof. 
\end{proof}

\begin{remark}\label{cross-morphism}
\normalfont
Let $(\Ac,T,\alpha)$ and  $(\Bc,T,\beta)$  be  ``real'' $C^*$-dynamical systems and $\psi\colon \Ac\to\Bc$ is an equivariant ``real'' morphism, then the corresponding $*$-morphism  
 $\psi\rtimes\iota \colon \Ac\rtimes_\alpha T\to \Bc\rtimes_\beta T$ 
 satisfying $((\psi\rtimes \iota)(f))(t)=\psi(f(t))$ for all $t\in T$ and $f\in\Cc_c(T,\Ac)$ is  a ``real'' morphism. 
\end{remark}

We now study the interaction between ``real'' structures and some constructions from \cite{Ri82}. 

\begin{definition}\label{WHdef}
\normalfont 
Let $(\Ac,\RR,\alpha)$ be a ``real'' $C^*$-dynamical system. 
We consider the $C^*$-algebra  $\cone \Ac:=\Cc_0(\RR\cup\{+\infty\},\Ac)$ 
with its $*$-morphism  $\ev_{+\infty}\colon \cone \Ac\to \Ac$, $f\mapsto f(+\infty)$ and the ideal $\susp \Ac:=\Ker(\ev_{+\infty})\simeq \Cc_0(\RR,\Ac)$. 
Then  $\cone \Ac$ is a ``real'' $C^*$-algebra with its ``real'' structure defined by $\overline{f}(t):=\overline{f(t)}$ for all $t\in\RR\cup\{+\infty\}$ and $f\in\cone \Ac$.
 Moreover $\susp \Ac$ is a ``real'' ideal, $\ev_{+\infty}$ is a ``real'' morphism, and we have the short exact sequence 
$$0\to\susp \Ac\hookrightarrow \cone \Ac\mathop{\longrightarrow}\limits^{\ev_{+\infty}} \Ac\to 0.$$
If we define 
$$\tau\otimes \alpha\colon \RR\to\Aut(\cone \Ac),\quad  ((\tau\otimes\alpha)_r f)(t):=\alpha_r(f(t-r))$$
then $(\cone A,\tau\otimes\alpha,\RR)$ is a ``real'' $C^*$-dynamical system and $\ev_{+\infty}$ intertwines the actions of $\RR$ on $\cone \Ac$ and $\Ac$ via $\tau\otimes\alpha$ and $\alpha$, respectuvely. 
In particular the ``real'' ideal $\susp \Ac$ is invariant to $\tau\otimes\alpha$. 
We further obtain the short exact sequence 
\begin{equation}\label{WHdef_eq1}
0\to\susp \Ac \rtimes_{\tau\otimes\alpha}\RR\hookrightarrow 
\cone \Ac\rtimes_{\tau\otimes\alpha}\RR\mathop{\longrightarrow}\limits^{\psi} \Ac\rtimes_\alpha\RR\to 0
\end{equation}
called the \emph{Wiener-Hopf extension} for $\Ac\rtimes_\alpha\RR$, 
where $\psi:=\ev_{+\infty}\rtimes\iota$ is a ``real'' morphism. 
(See Remark~\ref{cross-morphism}.)
\end{definition}

\begin{remark}\label{SvN}
\normalfont
In the special case $\Ac=\CC$ we get the ``real'' $C^*$-dynamical system $(\susp,\tau,\RR)$ with $\susp:=\susp\CC=\Cc_0(\RR)$ and $\tau\colon\RR\to\Aut(\susp)$, 
$(\tau_rf)(t):=f(t-r)$. 
If the regular representation of the group $\RR$ is again denoted by $\tau\colon \RR\to L^2(\RR)$, $(\tau_r\xi)(t):=\xi(t-r)$, 
and we define $M\colon\susp\to\Bc(L^2(\RR))$, $M(f)\xi=f\xi$ for all $f\in \susp$ and $\xi\in L^2(\RR)$, 
then we obtain a covariant representation $(M,\tau)$ of the $C^*$-dynamical system $(\susp,\tau,\RR)$ whose integrated representation gives a $*$-isomorphism 
\begin{equation}\label{SvN_eq1}
M\rtimes\tau\colon \susp\rtimes_\tau \RR\to\Kc(L^2(\RR)).
\end{equation} 
See \cite[Th. 4.24]{Wi07}. 
If $h\in\Cc_c(\RR)$ and $f\in\susp$, then the function  $h(\cdot)f$ (that is, $r\mapsto h(r)f$) belongs to $\Cc_c(\RR,\susp)\subseteq\susp\rtimes_\tau\RR$ 
and we have 
$$(M\rtimes\tau)(h(\cdot)f)=\int_{\RR}M(h(r)f)\tau(r)\de r=M(f)\int_{\RR}h(r)\tau(r)\de t$$
hence 
$$((M\rtimes\tau)(h(\cdot)f))\xi=f\cdot (h\ast\xi)\text{ for }\xi\in L^2(\RR)$$
The $*$-isomorphism~\eqref{SvN_eq1} is a ``real'' isomorphism. 
Here we regard $\Kc(L^2(\RR))$ as a ``real'' $C^*$-algebra 
with its ``real'' structure given by $\overline{T}:=CTC$ for all $T\in \Kc(L^2(\RR))$, where $C\colon L^2(\RR)\to L^2(\RR)$, $C(\xi):=\overline{\xi}$; 
thus, if $T\in \Kc(L^2(\RR))$ is an integral operator defined by an integral kernel $K_T\colon \RR\times\RR\to\CC$, 
then $\overline{T}$ is the integral operator defined by the integral kernel 
 $K_{\overline{T}}\colon \RR\times\RR\to\CC$, where $K_{\overline{T}}(t,r):=\overline{K_T(t,r)}$ for all $t,r\in\RR$. 
\end{remark}

\begin{proposition}\label{tsigns}
For every ``real'' $C^*$-dynamical system $(\Ac,\RR, \alpha)$ there exist  group isomorphisms 
\begin{align*}
\Theta_0 & \colon K_0(\Ac\rtimes_\alpha\RR)\to K_1(\Ac), \\
\Theta_1 & \colon K_1(\Ac\rtimes_\alpha\RR)\to K_0(\Ac), 
\end{align*}
satisfying 
\begin{align*}
\Theta_0(\overline{x})& =-\overline{\Theta_0(x)}
\text{ for all }x\in K_0(\Ac\rtimes_\alpha \RR), \\
\Theta_1(\overline{x})&=\quad \overline{\Theta_1(x)}
\text{ for all }x\in K_1(\Ac\rtimes_\alpha \RR).
\end{align*}
\end{proposition}

\begin{proof}
As proved in \cite{Ri82}, we have $K_0(\cone \Ac\rtimes_{\tau\otimes\alpha}\RR)=\{0\}$ and $K_1(\cone \Ac\rtimes_{\tau\otimes\alpha}\RR)=\{0\}$. 
Therefore, in the six-term exact sequence \eqref{hexagon} corresponding to the Wiener-Hopf extension~\eqref{WHdef_eq1}, 
the vertical arrows 
\begin{align*}
\delta_0 & \colon K_0(\Ac\rtimes_\alpha\RR)\to 
K_1(\susp \Ac \rtimes_{\tau\otimes\alpha}\RR), \\
\delta_1 & \colon K_1(\Ac\rtimes_\alpha\RR)\to 
K_0(\susp \Ac \rtimes_{\tau\otimes\alpha}\RR),
\end{align*}
are group isomorphisms. 

On the other hand, 
$(\susp \Ac,\tau\otimes\iota,\RR)$
is a ``real'' $C^*$-dynamical system and we have the $*$-isomorphism 
\begin{equation}
\label{tsigns_proof_eq1}
\gamma\colon \susp \Ac \rtimes_{\tau\otimes\alpha}\RR\to 
\susp \Ac \rtimes_{\tau\otimes\iota}\RR
\end{equation}
where 
$\gamma(f)\in\Cc_c(\RR,\susp \Ac)\subseteq \susp \Ac \rtimes_{\tau \otimes \iota}\RR$ is given by 
$\bigl((\gamma(f))(r)\bigr)(t)=\alpha_{-t}((f(r))(t))$,  $r,t\in\RR$,  
for every  $f\in\Cc_c(\RR,\susp \Ac)\subseteq \susp A \rtimes_{\tau\otimes\alpha}\RR$, 
(See \cite[page 147]{Ri82}.)
Then $\gamma$ is a ``real'' isomorphism. 

Moreover, by \cite[Lemma 2.75]{Wi07},
we have a $*$-isomorphism 
$$\eta\colon \susp \Ac \rtimes_{\tau \otimes\iota}\RR \to (\susp \Ac\rtimes_{\tau}\RR)\otimes \Ac$$
satisfying $\eta(h(\cdot)(f\otimes a))=(h(\cdot)f)\otimes a$ for all $h\in\Cc_c(\RR)$, $f\in\Cc_c(\RR)\subseteq\susp$, and $a\in \Ac$, 
where we regard $h(\cdot)f$ as an element of $\Cc_c(\RR,\susp)$ as in Remark~\ref{SvN}. 
Taking into account the $*$-isomorphism $M\rtimes\tau$ from \eqref{SvN_eq1}, 
we further obtain the $*$-isomorphism
\begin{equation}
\label{tsigns_proof_eq2}
\kappa:=((M\rtimes\tau)\otimes\id_\Ac)\circ \eta \colon \susp \Ac \rtimes_{\tau\otimes\iota}\RR 
\to \Kc(L^2(\RR))\otimes \Ac
\end{equation}
satisfying 
$\kappa(h(\cdot)(f\otimes a))=\bigl((M\rtimes\tau)(h(\cdot)f)\bigr)\otimes a$ 
for all $h\in\Cc_c(\RR)$, $f\in\Cc_c(\RR)\subseteq\susp$, and $a\in A$. 
In particular, this shows that $\Kc(L^2(\RR))\otimes \Ac$ has the structure of a ``real'' $C^*$-algebra satisfying $\overline{T\otimes a}=\overline{T}\otimes\overline{a}$ for all $T\in\Kc(L^2(\RR))$ 
(cf. the end of Remark~\ref{SvN}) and $a\in \Ac$.
Then the above $*$-isomorphism $\kappa$ is a ``real'' isomorphism. 
 
Using \eqref{tsigns_proof_eq1} and \eqref{tsigns_proof_eq2}, 
we now obtain the ``real'' isomorphism 
\begin{equation*}
\kappa\circ\gamma\colon \susp \Ac \rtimes_{\tau\otimes\alpha}\RR\to 
\Kc(L^2(\RR))\otimes \Ac.
\end{equation*}
This in turn gives the group isomorphisms
$$K_j(\kappa\circ\gamma)\colon K_j(\susp \Ac \rtimes_{\tau\otimes\alpha}\RR)
\to 
K_j(\Kc(L^2(\RR))\otimes \Ac)
$$
satisfying $K_j(\kappa\circ\gamma)(\overline{x})=\overline{K_j(\kappa\circ\gamma)(x)}$ 
for all $x\in K_j(\susp \Ac \rtimes_{\tau\otimes\alpha}\RR)$ and $j=0,1$. 

Finally, we select any $\xi_0\in L^2(\RR)$ with $\overline{\xi_0}=\xi_0$ and $\Vert \xi_0\Vert=1$ and we consider its corresponding rank-one projection $p_0:=(\cdot\mid\xi_0)\xi_0\in\Kc(L^2(\RR))$, 
so that $\overline{p_0}=p_0$ in the ``real'' $C^*$-algebra $\Kc(L^2(\RR))$. 
Then the mapping 
$$\mu_{p_0}\colon \Ac\to \Kc(L^2(\RR))\otimes \Ac,\quad a\mapsto p_0\otimes a$$
is a ``real'' morphism 
hence the group morphism 
$$K_j(\mu_{p_0})\colon K_j(\Ac)\to K_j(\Kc(L^2(\RR))\otimes \Ac)$$ 
satisfies $K_j(\mu_{p_0})(\overline{y})=\overline{K_j(\mu_{p_0})(y)}$ for all $y\in K_j(\Ac)$ and $j=0,1$. 
On the other hand, $K_j(\mu_{p_0})$ is actually a group isomorphism for $j=0,1$. 
(See \cite[6.4.1 and 8.2.8]{RLL00}.)
Consequently we obtain the group isomorphisms 
$$\Theta_0:=K_1(\mu_{p_0})^{-1}\circ K_1(\kappa\circ\gamma) \circ\delta_0
\colon K_0(\Ac\rtimes_\alpha\RR)\to K_1(\Ac)$$
and 
$$\Theta_1:=K_0(\mu_{p_0})^{-1}\circ K_0(\kappa\circ\gamma) \circ\delta_1
\colon K_1(\Ac\rtimes_\alpha\RR)\to K_0(\Ac).$$
Lemma~\ref{deltas}  ensures that $\Theta_1$ and $\Theta_2$ have the required properties. 
\end{proof}

\begin{corollary}\label{signs_semid}
Let $N$ be a locally compact group and $\alpha\colon \RR\to\Aut(N)$ be a 
continuous action of $\RR$ by automorphisms of $N$. 
Then  there exist group isomorphisms 
\begin{align*}
\Theta_0 & \colon K_0(C^*(N\rtimes_\alpha\RR))\to K_1(C^*(N)), \\
\Theta_1 & \colon K_1(C^*(N\rtimes_\alpha\RR))\to K_0(C^*(N))
\end{align*} 
satisfying 
\begin{align*}
\Theta_0(\overline{x})
&=-\overline{\Theta_0(x)}\text{ for all }x\in K_0(C^*(N\rtimes_\alpha\RR)),\\
\Theta_1(\overline{x})
&=\quad \overline{\Theta_1(x)}\text{ for all }x\in K_1(C^*(N\rtimes_\alpha\RR)). 
\end{align*}
\end{corollary}

\begin{proof}
There exists a group morphism $\beta\colon \RR\to \Aut(C^*(N))$ for which 
$(C^*(N),\RR,\beta)$ is a ``real'' $C^*$-dynamical system,  and  
the natural inclusion map 
$$\Cc_c(\RR,\Cc_c(N))\hookrightarrow\Cc_c(N\times \RR)$$
extends to a $*$-isomorphism $\gamma\colon C^*(N)\rtimes_\beta\RR\to C^*(N\rtimes_\alpha\RR)$, 
by \cite[Prop. 3.11]{Wi07}.
The above inclusion map intertwines the operation of taking the complex-conjugates of the functions on $\RR$, $N$, and $N\times\RR$, 
hence $\gamma$ is a ``real'' isomorphism. 
Then $K_j(\gamma)\colon K_j(C^*(N)\rtimes_\beta\RR)\to K_j(C^*(N\rtimes_\alpha\RR))$ is a group isomorphism satisfying 
$K_j(\gamma)(\overline{x})=\overline{K_j(\gamma)(x)}$ for all $x\in K_j(C^*(N)\rtimes_\beta\RR)$ and $j=0,1$. 
Now the assertion follows by an application of Proposition~\ref{tsigns}. 
\end{proof}

\begin{corollary}
\label{signs_solvable}
Let $G$ be a solvable Lie group and denote $n:=\dim G$.  
Then the following assertion hold: 
\begin{enumerate}[{\rm(i)}]
	\item\label{signs_solvable_item1} If $n\in 2\ZZ$ then $K_1(C^*(G))= \{0\}$, $K_0(C^*(G))\simeq\ZZ$, and for every $x\in K_0(C^*(G))$ we have 
	$$\overline{x}=\begin{cases}
	x &\text{ if }n\in 4\ZZ,\\
	-x &\text{ if }n\in 4\ZZ+2.
	\end{cases}$$
	\item\label{signs_solvable_item2}  If $n\in 2\ZZ+1$ then $K_0(C^*(G))=\{0\}$, $K_1(C^*(G))\simeq\ZZ$, and for every $x\in K_1(C^*(G))$ we have 
	$$\overline{x}=\begin{cases}
	x &\text{ if }n\in 4\ZZ+1,\\
	-x &\text{ if } n\in 4\ZZ+3.
	\end{cases}$$
\end{enumerate}
\end{corollary}

\begin{proof}
The group isomorphisms from the statement are well known.  (See \cite[Sect.~V, Cor.~7]{Co81}.)
To prove the assertions on $\overline{x}$  
we recall that, since $G$ is a 
solvable Lie group, 
there exists a Lie group isomorphism $G\simeq G_1\rtimes\RR$ for a suitable 
solvable Lie group $G_1$.
Now the conclusion follows by induction, using Corollary~\ref{signs_semid} and Lemma~\ref{dim1}. 
\end{proof}

\subsection{Continuous fields of $C^*$-algebras }

The following lemma is implicitly used in the proof of \cite[Thm.~3.1]{ENN93}. 

\begin{lemma}\label{lemma-cf1}
Let $((\Ac_t)_{t\in S}, \Theta)$ be a continuous field of $C^*$-algebras over the locally compact space $S$.
Assume that for a $t_0\in S$ there is a projection $p_0\in \Pg(\Ac_{t_0})\setminus \{0\}$. 
Then there is an open neighbourhood $V_0$ of $t_0$ in $S$ and a section 
$\theta_0\in \Theta\vert_{V_0}$ such that $\theta_0(t_0) = p_0$ and  $ \theta_0(t) \in  \Pg(\Ac_{t})\setminus \{0\}$
for every $t\in V_0$. 
\end{lemma}

\begin{proof}
For $\delta \in (0, 1/2)$, define
$$ U=\{ z\in \CC\mid |z|< \delta\} \cup \{z\in \CC\mid |z-1|< \delta\}.$$
Then $\opn{Sp}'_{A_{t_0}}(p_0)=\{0, 1\}\subset U$. 
(Here $\opn{Sp}'_{\Ac}(a)$ denotes the 
spectrum of $a$ in the non-necessarily unital $C^*$-algebra $\Ac$; see \cite[1.1.6]{Di64}.)
Then, by \cite[20.1.10]{Di64},  there exists $x_1\in \Theta$ such that $x_1(t_0)=p_0$.
Define $x_2 = \frac{1}{2} (x_1 +x_1^*)$; 
then $x_2 \in \Theta$, $x_2= x_2^*$ and $x_2(t_0)= p_0$. 
It follows
by \cite[10.3.6]{Di64} that there is an open neighbourhood $V_0$ of $t_0$ in $S$ such that 
$\opn{Sp}'_{A_{t}}(x_2(t))\subset U$ for every $t\in V_0$. 
Hence, if  $f\in \Cc (\CC)$ is such that $f(t)=1$ for $t\in \{z\in \CC\mid |z-1|< \delta\}$ and 
$f(t) =0$ for $t \in \{ z\in \CC\mid |z|< \delta\}$, then $f(x_2(t)) \in \Pg(\Ac_t)$.
By \cite[10.3.3]{Di64}, $\theta_0= f(x_2)\in \Theta$.
Since the function $\Vert \theta_0(\cdot)\Vert$ is continuous on $V_0$ and $\Vert \theta_0(t)\Vert
\in \{0, 1\}$, it follows that $\Vert \theta_0(t)\Vert = \Vert \theta_0(t_0)\Vert=1$ for every 
$t \in  V_0$. 
We have thus obtained that $\theta_0(t) \in \Pg(\Ac_t) \setminus \{ 0\}$ for every  $ t \in V_0$, $\theta_0(t_0)= p_0$, hence $\theta_0$ satisfies all the conditions in the statement.
\end{proof}

\begin{proposition}\label{prop-cf2}
Let $((\Ac_t)_{t\in [0, 1]}, \Theta)$ be a continuous field of $C^*$-algebras,  trivial away from $0$ (that is, trivial on $(0, 1]$). 
If $\Ac_t$ is stably finite for $t\in (0, 1]$, then $\Ac_0$ is stably finite.
\end{proposition}

\begin{proof}
Assume that $\Ac_0$ is not stably finite. 
Then by Proposition~\ref{P1} \eqref{P1_item1} it follows that there is $k\ge 1$  and 
$p_0 \in \Pg(M_k(\Ac_0))\setminus \{\0_k\}$ such that $[p_0] =0 \in K_0(\Ac)$. 
Since $(M_k(\Ac_t))_{t\in [0, 1]}$ is a continuous field of $C^*$-algebras, trivial away from $0$ 
(see \cite[Thm.~2.4]{ENN93}), we may assume that $k=1$. 
By Lemma~\ref{lemma-cf1} and since $(\Ac_t)_{t\in [0, 1]}$ is trivial away from $0$, there is $\theta_0\in \Theta$ such that $\theta_0(0)=p_0$ and 
$\theta_0(t)\in \Pg(\Ac_t)\setminus \{0 \}$ for every $t \in [0, 1]$. 
It follows by \cite[Thm.~3.1 and its proof]{ENN93} that there is a group homomorphism 
$\varphi \colon K_0(\Ac_0) \to K_0(\Ac_t)$ such that 
$\varphi([p_0])= [\theta_0(1)]$. 
Since we have assumed that $[p_0]=0$, we get that for  $\theta_0(1)\in \Pg(\Ac_1)\setminus \{0\}$
 we have $[\theta_0(1)]=0$, thus by Proposition~\ref{P1} \eqref{P1_item1}, $\Ac_1$ is not stably finite. 
 This is a contradiction; thus $\Ac_0$ must be stably finite.
\end{proof}

\subsection{On open points in the primitive ideal spectrum}

\begin{proposition}\label{P3}
	Let $\Ac$ be a separable $C^*$-algebra
	If $\pi_0\colon\Ac\to\Bc(\Hc_0)$ is a $*$-representation  
	with its kernel $\Pc_0:=\Ker\pi_0\subseteq\Ac$ 
	and $\Kc(\Hc_0)\subseteq\pi_0(\Ac)\ne\{0\}$, 
	then the following conditions are equivalent: 
	\begin{enumerate}[{\rm(i)}]
		\item\label{P3_item1} $\{\Pc_0\}$ is an open subset of $\Prim(\Ac)$. 
		\item\label{P3_item2} There exists a closed two-sided ideal $\Jc_0\subseteq\Ac$ for which 
		$\pi_0\vert_{\Jc_0}\colon \Jc_0\to\Kc(\Hc_0)$ is a $*$-isomorphism.
		\end{enumerate}
		If these conditions are satisfied, then 
		 \begin{equation}\label{P3_proof_eq7}
		\Jc_0=\bigcap\limits_{\Pc\in\Prim(\Ac)\setminus\{\Pc_0\}}\Pc
		\end{equation}
		and moreover $\Jc_0$ is a minimal closed two-sided ideal of $\Ac$ 
		with 	$\Pc_0\cap\Jc_0=\{0\}$. 
\end{proposition}

\begin{proof}
The hypothesis $\Kc(\Hc_0)\subseteq\pi_0(\Ac)\ne\{0\}$ implies that the $*$-representation $\pi_0$ is irreducible and $\Hc_0\ne\{0\}$. 
Then, since $\Ac$ is separable, the Hilbert space $\Hc_0$ is separable, too.
We will show that both conditions in the statement are equivalent to the following: 
\begin{enumerate}[{\rm(i)}]
    \setcounter{enumi}{2}
	\item\label{P3_item3} There exists a closed two-sided ideal $\Jc_0\subseteq\Ac$ such that 
	\begin{equation}\label{P3-eq1}
	\Prim(\Ac)=\{\Pc_0\}\sqcup\{\Pc\in\Prim(\Ac)\mid \Jc_0\subseteq\Pc\}.
	\end{equation}
\end{enumerate}
	
\eqref{P3_item1}$\iff$\eqref{P3_item3}: 
This and  \eqref{P3_proof_eq7} follow by the definition of the topology of $\Prim(\Ac)$. 
(See \cite[3.1.1]{Di64}.) 

\eqref{P3_item3}$\implies$\eqref{P3_item2}: 
The hypothesis \eqref{P3-eq1} implies $\Jc_0\not\subseteq\Pc_0=\Ker\pi_0$, 
that is, $\pi_0\vert_{\Jc_0}\ne0$. 
Moreover, for every irreducible $*$-representation $\pi\colon\Ac\to\Bc(\Hc)$ we have 
$$\pi\vert_{\Jc_0}\ne0\iff\Jc_0\not\subseteq\Ker\pi
\mathop{\iff}\limits^{\eqref{P3-eq1}}\Ker\pi=\Pc_0
\iff[\pi]=[\pi_0]\in\widehat{\Ac}$$
where the last equivalence follows by \cite[Cor. 4.1.10]{Di64} 
since $\Kc(\Hc_0)\subseteq\pi_0(\Ac)$. 
Then, by \cite[Prop. 2.10.4]{Di64}, $\widehat{\Jc_0}$ consists of only one point, namely $\widehat{\Jc_0}=\{[\pi_0\vert_{\Jc_0}]\}$. 
Since $\Ac$ is separable it follows that $\Jc_0$ is separable, too. 
Then $\Jc_0$ is $*$-isomorphic to the $C^*$-algebra of all compact operators on a separable complex Hilbert space by \cite[4.7.3]{Di64}. 
Now, since $\pi_0\vert_{\Jc_0}\ne0$ and $\pi_0$ is an irreducible representation of~$\Ac$, hence $\pi_0\vert_{\Jc_0}\colon \Jc_0\to\Bc(\Hc_0)$  is an irreducible representation, it follows that 
$\pi_0\vert_{\Jc_0}\colon \Jc_0\to\Kc(\Hc_0)$ is a $*$-isomorphism. 
(See \cite[Cor. 4.1.5]{Di64}.) 
Hence \eqref{P3_item3} holds true with $\Jc_1:=\Jc_0$.

\eqref{P3_item2}$\implies$\eqref{P3_item3}: 
The hypothesis \eqref{P3_item2} implies $\widehat{\Jc_0}=\{[\pi_0\vert_{\Jc_0}]\}$. 
Then, 
for every irreducible $*$-representation $\pi\colon\Ac\to\Bc(\Hc)$, we have either $\pi\vert_{\Jc_0}=0$ or $[\pi\vert_{\Jc_0}]=[\pi_0\vert_{\Jc_0}]\in\widehat{\Jc_1}$. 
That is, either $\Jc_1\subseteq\Ker\pi$ or $[\pi]
=[\pi_0]\in\widehat{\Ac}$  by \cite[Prop. 2.10.4]{Di64}. 
Thus
$\Prim(\Ac)=\{\Ker\pi_0\}\sqcup\{\Pc\in\Prim(\Ac)\mid \Jc_1\subseteq\Pc\}$, 
hence~\eqref{P3_item2} holds true. 

Finally, if \eqref{P3_item1}--\eqref{P3_item3} hold true, then $\Jc_0$ is $*$-isomorphic to $\Kc(\Hc_0)$, hence $\Jc_0$ is a simple $C^*$-algebra, and then it is also a minimal ideal of $\Ac$ with $\Pc_0\cap\Jc_0=\{0\}$.
\end{proof}

\begin{remark}\label{R5}
\normalfont
In Proposition~\ref{P3} we have the short exact sequence 
$$0\to\Pc_0\hookrightarrow\pi_0^{-1}(\Kc(\Hc_0))\mathop{\longrightarrow}\limits^{\pi_0}\Kc(\Hc_0)\to0$$
and this extension is trivial in the sense that 
$\pi_0\vert_{\pi_0^{-1}(\Kc(\Hc_0))}$ has a right inverse, namely $(\pi_0\vert_{\Jc_0})^{-1}\colon\Kc(\Hc_0)\to\Jc_0$ given by  Proposition~\ref{P3}\eqref{P3_item3}. 
This also shows the direct sum decomposition of ideals 
$\pi_0^{-1}(\Kc(\Hc_0))=\Pc_0\dotplus\Jc_0$. 
\end{remark}

\begin{remark}\label{P3_group}
\normalfont 
The hypothesis $\Kc(\Hc_0)\subseteq\pi_0(\Ac)$ in Proposition~\ref{P3} is superfluous if  $\Ac=C^*(G)$ for an exponential Lie group~$G$.   
In fact, let $\pi_0\colon G\to\Bc(\Hc_0)$ be an irreducible unitary representation 
with its corresponding irreducible $*$-representation $\pi_0\colon \Ac\to\Bc(\Hc_0)$ with $\Pc_0:=\Ker\pi\subseteq\Ac$.  
Since $G$ is type~I, we have $\Kc(\Hc_0)\subseteq\pi_0(\Ac)$, 
and on the other hand $\{\Pc_0\}$ is an open subset of $\Prim(\Ac)$ if and only if the unitary representation~$\pi_0$ is square integrable, 
by \cite[Prop. 2.3 and 2.14]{Ros78} and \cite[Cor. 2]{Gr80}. 
Moreover the irreducible unitary representation~$\pi_0$ is square integrable if and only if its corresponding coadjoint orbit is open in~$\gg^*$ by 
\cite[Thm.~3.5]{Ros78}. 
This provides an alternative argument for the fact that the Kirillov-Bernat correspondence gives a bijection between the open points of $\Prim(G)$ and the open coadjoint orbits of~$G$, without using the more difficult and deep fact that the Kirillov-Bernat map is actually a homeomorphism. 
\end{remark}

\begin{corollary}\label{C4}
Let $\Ac$ be a $C^*$-algebra and, for $k=1,\dots,n$, let $\pi_k\colon\Ac\to\Bc(\Hc_j)$ be a $*$-representation satisfying  the hypotheses of Proposition~\ref{P3}, with its corresponding ideal $\Jc_k\subseteq\Ac$ for which 
$\pi_k\vert_{\Jc_k}\colon \Jc_k\to\Kc(\Hc_k)$ is a $*$-isomorphism, 
and $\Pc_k:=\Ker\pi_k$. 
Then the following assertions hold: 
\begin{enumerate}[{\rm(i)}]
\item\label{C4_item1} 
We have $\Jc_{k_1}=\Jc_{k_2}$ if and only if $\Pc_{k_1}=\Pc_{k_2}$. 
\item \label{C4_item2} 
If we assume $\Pc_{k_1}\ne\Pc_{k_2}$ for $k_1\ne k_2$, then
\begin{equation}\label{C4_item2_eq1} 
	\text{$\Jc:=\Jc_1+\dots+\Jc_n$ is a direct sum of ideals of $\Ac$}
\end{equation}
and
\begin{equation}\label{C4_item2_eq2}
	\Prim(\Ac)=\{\Pc_1\}\sqcup\cdots\sqcup\{\Pc_n\}\sqcup \{\Pc\in\Prim(\Ac)\mid\Jc\subseteq\Pc\}.
\end{equation}
\end{enumerate}
\end{corollary}

\begin{proof}
\eqref{C4_item1} 
We have $\Prim(\Ac)\setminus\{\Pc_k\}=\{\Pc\in\Prim(\Ac)\mid \Jc_k\subseteq\Pc\} $ by \eqref{P3-eq1}, 
hence $\Jc_{k_1}=\Jc_{k_2}$ implies $\Pc_{k_1}=\Pc_{k_2}$. 
Conversely,  by \eqref{P3_proof_eq7}, 
$\Pc_{k_1}=\Pc_{k_2}$ implies $\Jc_{k_1}=\Jc_{k_2}$. 

\eqref{C4_item2} 
If $k_1\ne k_2$, then $\Pc_{k_1}\ne\Pc_{k_2}$ by hypothesis, hence $\Jc_{k_1}\ne\Jc_{k_2}$ by \eqref{C4_item1}. 
Since both $\Jc_{k_1}$ and $\Jc_{k_2}$ are distinct minimal ideals of $\Ac$ 
and $\Jc_{k_1}\Jc_{k_2}\subseteq\Jc_{k_1}\cap \Jc_{k_2}$, 
we obtain $\Jc_{k_1}\Jc_{k_2}=\Jc_{k_1}\cap \Jc_{k_2}=\{0\}$,  and then 
\eqref{C4_item2_eq1}  is straightforward.

We now prove \eqref{C4_item2_eq2}. 
In fact, by \eqref{P3-eq1}, 
%{P3_proof_eq1} 
we have 
$\Prim(\Ac)\setminus\{\Pc_k\}=\{\Pc\in\Prim(\Ac)\mid \Jc_k\subseteq\Pc\}$ for $k=1,\dots,n$, hence 
\allowdisplaybreaks
\begin{align*}
\Prim(\Ac)\setminus\{\Pc_1,\dots,\Pc_n\}
&=\bigcap_{k=1}^n\Prim(\Ac)\setminus\{\Pc_k\}  
=\bigcap_{k=1}^n \{\Pc\in\Prim(\Ac)\mid \Jc_k\subseteq\Pc\} \\
&=\{\Pc\in\Prim(\Ac)\mid \Jc_1+\cdots+\Jc_n\subseteq\Pc\}.
\end{align*}
This finishes the proof. 
\end{proof}

The  next result is needed in the proof of  Corollary~\ref{cf-cor8}.
 
\begin{proposition}\label{proj}
	Assume the setting of Proposition~\ref{P3} 
	and, additionally, that 
	\begin{enumerate}[{\rm(i)}]
	 \item\label{proj_item1} $\Ac$ is a ``real'' $C^*$-algebra; 
	 \item\label{proj_item2} $\Pc_0\cap\overline{\Pc_0}=\{0\}$; 
	 \item\label{proj_item3} if $p\in\Jc_0$ is a minimal projection, then 
	 $K_0(\Ac)=\{n[p]_0\mid n\in\ZZ\}$; 
	 \item\label{proj_item4} $\Pg(A)\setminus(\Jc_0+\overline{\Jc_0})\ne\emptyset$. 
	\end{enumerate}
Then $\Ac$ is not stably finite. 
\end{proposition}

\begin{proof}
We define
$\overline{\pi_0}\colon\Ac\to\Bc(\Hc_0)$, $\overline{\pi_0}(a):=C\pi_0(\overline{a})C$, 
for a fixed antilinear involutive isometry $C\colon \Hc\to\Hc$. 
Then $\overline{\pi_0}$ is an irreducible $*$-representation 
with $\Ker\overline{\pi_0}=\overline{\Pc_0}$ and $\{\overline{\Pc_0}\}$ is an open subset of $\Prim(\Ac)$. 
Moreover, using~\eqref{P3_proof_eq7}, one can show that 
$\Jc_{\overline{\pi_0}}=\overline{\Jc_0}$, where $\Jc_{\overline{\pi_0}}$ is the minimal ideal of $\Ac$ that is given by Proposition~\ref{P3} for the representation $\overline{\pi_0}$. 
By Remark~\ref{R5}, we then obtain 
\begin{equation}\label{proj_proof_eq1}
\overline{\pi_0}^{-1}(\Kc(\Hc_0))
=\Jc_{\overline{\pi_0}}\dotplus \Ker\overline{\pi_0}
=\overline{\Jc_0}\dotplus \overline{\Ker\pi_0} 
=\overline{\pi_0^{-1}(\Kc(\Hc_0))}.
\end{equation}
Since $\Pc_0\cap\overline{\Pc_0}=\{0\}$ by hypothesis, we may use Corollary~\ref{C4} for the representations $\pi_0$ and $\overline{\pi_0}$, and we thus obtain $\Jc_0\cdot\overline{\Jc_0}=\{0\}$. 

Now let us denote $\Jc:=\Jc_0\dotplus \overline{\Jc_0}$ 
and select any $q\in \Pg(A)\setminus\Jc$. 
We show that 
\begin{equation}\label{proj_proof_eq2}
\dim(\pi_0(q)\Hc)=\infty\text{ and }\dim(\overline{\pi_0}(q)\Hc)=\infty.
\end{equation}
In fact, since $\pi_0(q),\overline{\pi_0}(q)\in\Bc(\Hc_0)$ are projections, 
it suffices to show that $\pi_0(q)\not\in \Kc(\Hc_0)$, and this will imply $\overline{\pi_0}(q)\not\in\Kc(\Hc_0)$ by~\eqref{proj_proof_eq1}. 
We argue by contradiction: 
Assuming $\pi_0(q)\in \Kc(\Hc_0)$, we obtain $\overline{\pi_0}(q)\in\Kc(\Hc_0)$  by~\eqref{proj_proof_eq1} and then,  by Proposition~\ref{P3} there exist uniquely determined projections $p_0\in\Pg(\Jc_0)$ and $r_0\in\Pg(\overline{\Jc_0})$ with $\pi_0(p_0)=\pi_0(q)$ and $\overline{\pi_0}(r_0)=\overline{\pi_0}(q)$. 
Since $\Pc_0\ne\overline{\Pc_0}$ by the hypothesis~\eqref{proj_item2}, 
we have 
$\Jc_0\subseteq\overline{\Pc_0}$ and $\overline{\Jc_0}\subseteq\Pc_0$ by \eqref{P3-eq1} in Proposition~\ref{P3}, and we then obtain 
$(\pi_0\oplus\overline{\pi_0})(p_0+r_0)=\pi_0(q)\oplus \overline{\pi_0}(q)
=(\pi_0\oplus\overline{\pi_0})(q)$. 
The hypothesis $\Pc_0\cap\overline{\Pc_0}=\{0\}$ then implies $q=p_0+r_0\in\Jc_0\dotplus \overline{\Jc_0}=\Jc$, which is a contradiction with the way $q$ was selected. 
Thus \eqref{proj_proof_eq2} is proved. 

Now, if $p\in\Jc_0$ is a minimal projection, it follows by the hypothesis that there exists $n\in\ZZ$ with $[q]_0=n[p]_0\in K_0(\Ac)$. 
There are three possible cases: 

Case 1: $n=0$. 
Then $[q]_0=0\in K_0(\Ac)$, and Proposition~\ref{P1} shows that $\Ac$ is not stably finite. 

Case 2: $n<0$. 
Then, denoting $k:=\vert n\vert$, we have 
$$0=[q]_0+k[p]_0=[q\oplus\underbrace{p\oplus\cdots\oplus p}_{k\text{ times}}]_0$$
hence, since $q\oplus p\oplus\cdots\oplus p\in M_{k+1}(\Ac)\setminus\{0\}$, 
 Proposition~\ref{P1} again shows that $\Ac$ is not stably finite. 
 
 Case 3: $n>0$. 
 In this case, by \eqref{proj_proof_eq2}, there exists $\widetilde{p}_1\in\Pg(\Kc(\Hc_0))$ with $\widetilde{p}_1\le\pi(q)$ and $\dim(\widetilde{p}_1(\Hc))=n$. 
 By Proposition~\ref{P3}, there exists a unique $p_1\in\Pg(\Jc_0)$ with $\pi_0(p_1)=\widetilde{p}_1$. 
 We already noted above that $\Jc_0\subseteq\overline{\Pc_0}=\Ker\overline{\pi_0}$, 
 hence $\overline{\pi_0}(p_1)=0$, and then 
 $$(\pi_0\oplus\overline{\pi_0})(p_1)=\pi_0(p_1)\oplus 0=\widetilde{p}_1\oplus 0\le\pi_0(q)\oplus\overline{\pi_0}(q)=(\pi_0\oplus\overline{\pi_0})(q).$$
 As above, the hypothesis hypothesis $\Pc_0\cap\overline{\Pc_0}=\{0\}$ then implies 
 $p_1\le q$, hence $q-p_1\in\Pg(\Ac)$ and $p_1(q-p_1)=0$. 
 Now, by \cite[3.1.7(iv)]{RLL00}, we obtain 
 \begin{equation}\label{proj_proof_eq3}
 [q]_0=[p_1]_0+[q-p_1]_0\in K_0(\Ac)\subseteq K_0(\widetilde{\Ac}).
 \end{equation}
 On the other hand, since  $\dim(\widetilde{p}_1(\Hc))=n$ and $\pi_0\vert_{\Jc_0}\colon\Jc_0\to\Kc(\Hc_0)$ is a $*$-iso\-mor\-phism, we obtain $[p_1]_0=n[p]_0$ in $K_0(\Jc_0)$. 
 Denoting by $\varphi\colon \Jc_0\to\Ac$ the inclusion map, 
 it then follows that 
 $K_0(\varphi)([p_1]_0)=nK_0(\varphi)([p]_0)$ in $K_0(\Ac)$, 
 that is, $[p_1]_0=n[p]_0$ in $K_0(\Ac)$. 
 Then, using \eqref{proj_proof_eq3} and the way $n$ was chosen, we obtain 
 $[q-p_1]_0=0\in K_0(\Ac)$. 
 On the other hand, $q-p_1\ne 0$ since $p_1\in\Jc_0\subseteq\Jc$, while $q\in\Pc(\Ac)\setminus\Jc$. 
 We may thus apply Proposition~\ref{P1} to obtain that $\Ac$ is not stably finite. 
\end{proof}

\section{$C^*$-algebras of exponential Lie groups with open coadjoint orbits}\label{section4k}

This section contains some of our  results on the relation between the quasi-compact open subsets in the primitive ideal space of the $C^*$-algebra of a solvable Lie group and the finite approximation properties of that $C^*$-algebra (Corollaries \ref{solv-4n+2}~and~\ref{cf-cor8}).  
These results mostly concern the exponential Lie groups that admit open coadjoint orbits. 
However we start with some results on general solvable Lie groups.

\subsection{Solvable Lie groups of dimension $\not \in 4 \ZZ$}
\begin{theorem}\label{4n+2}
Let $G$ be 
a  solvable Lie group  with $\dim G \not \in 4 \ZZ$. 
 Then  $C^*(G)$ is stably finite if and only if it is 
 is stably projectionless.
 \end{theorem}

\begin{proof}
The fact that if $C^*(G)$  is stably projectionless then it is stably finite follows from Corollary~\ref{rem-1.5} 
\eqref{rem-1.5_i}. 

For the reverse implication assume first that $\dim G$ is odd. 
Recall from Corollary~\ref{signs_solvable} 
that  $K_i(C^*(G))= K_i(\RR^{\dim G})$, $i =0, 1$. 
Hence, if $\dim G$ is odd, we have $K_0(C^*(G))=\{0\}$.
Then the statement
is a direct consequence of Lemma~\ref{rem-1.5}.

It remains to analyze the case when $\dim G\in 4 \ZZ +2$. 
We prove  that  if  $C^*(G)$ is not stably projectionless then it is not stably finite. 
Let $0\ne p\in \Pg_k(C^*(G))$. 
Then 
$ [p]_0+[\overline{p}]_0=  [p]_0+\overline{[p]}_0 = 0$
by Corollary~\ref{signs_solvable}\eqref{signs_solvable_item1}.
If $p=\overline{p}$  it follows that $[p]_0=0$, hence $C^*(G)$ is not stably finite, 
by Proposition~\ref{P1}.
If $p \ne  \overline{p}$, define  $ q:= p\oplus \overline{p}=\opn{diag}(p, \overline{p}) \in \Pg_{2k} (C^*(G))\setminus\{0\}$.
Then $[q]_0= [p]_0+[\overline{p}]_0= 0$ as above by Corollary~\ref{signs_solvable}\eqref{signs_solvable_item1}, hence, again by Proposition~\ref{P1},  $C^*(G)$ is not stably finite. 
\end{proof}

\begin{corollary}\label{solv-4n+2}
Let $G$  be an exponential solvable Lie group with $\dim G \in 4\ZZ+2$. 
If $G$ has  open coadjoint orbits, 
then $C^*(G)$ is not stably finite. 
\end{corollary}

\begin{proof}
For  an exponential Lie group $G$, an open coadjoint orbit corresponds to an open point $[\pi]\in \widehat{G}\simeq\Prim(G)$.  (See Remark~\ref{P3_group}.) Moreover, since $G$ is separable and type I, $\pi$ is square integrable and 
$\pi(C^*(G))$ contains the compact operators, by \cite[Cor. 1 and 2]{Gr80} and \cite[Prop.~2.3]{Ros78}.
Then by Proposition~\ref{P3} there is a minimal ideal $\Jc_0\subseteq C^*(G)$  such that $\Jc_0\simeq \Kc(\Hc_0)$ for a Hilbert space $\Hc_0$. 
Hence  there exists  $p \in \Jc_0$, $0\ne p= p^*= p^2$.  
The corollary now follows from Theorem~\ref{4n+2}.
\end{proof}

\subsection{Groups of the form $N \rtimes \RR$.}

We are going to see that the above result of Theorem~\ref{4n+2} fails to be true for groups of dimension of the form 
$4 k$, $k\in \NN$. 
To show this,  to give a simple necessary condition for stably finiteness, and to study a little bit further the case 
$\dim G\in 4 \ZZ +2$, we  restrict ourselves to the the groups of the form $G= N \rtimes \RR$, where  $N$ is  a nilpotent Lie group.
 
But first we consider the case of groups $G=N\rtimes \RR$ where $N$ is abelian, 
that is, the case of 
the  generalized $ax+b$-groups, where there is a quite clean  relation
 between quasi-compact open sets and finite approximation properties.

\subsubsection{The case of the generalized $ax+b$ groups}
Most of the following example is already known (see \cite{GKT92}, \cite[Th. 2.15]{BB18b}, and \cite[Ex.~4.8]{BB21}).

\begin{example}[generalized $ax+b$-groups]\label{ax+b} \normalfont 
Let $\Vc$ be a finite-dimensional real vector space, $D\in\End(\Vc)$, 
and $G_D:=\Vc\rtimes_{\alpha_D}\RR$ their corresponding semidirect product, called generalized $ax+b$-group. 
We recall from \cite[\S 2]{BB18b} the notation 
$\alpha_D\colon\RR\to\End(\Vc)$, 
$\alpha_D(t):=\ee^{tD}$, 
so the group operation in $G_D$ is given by $(v_1,t_1)\cdot(v_2,t_2)=(v_1+\ee^{t_1D}v_2,t_1+t_2)$ for all $v_1,v_2\in\Vc$ and $t_1, t_2\in\RR$. 
Then  we claim that the following assertions are equivalent: 
\begin{enumerate}[{\rm(i)}]
\item\label{ax+b_item1} Either $\mathrm{Re}\, z>0$ for every $z\in\sigma(D)$ or $\mathrm{Re}\, z<0$ for every $z\in\sigma(D)$. 
\item\label{ax+b_item2} The $C^*$-algebra $C^*(G_D)$ is not quasidiagonal. 
\item\label{ax+b_item3} The $C^*$-algebra $C^*(G_D)$ is not AF-embeddable. 
\item\label{ax+b_item4} There exists a nonempty quasi-compact open subset of $\widehat{G_D}$. 
\item\label{ax+b_item5} There exists a nonempty quasi-compact open subset of $\Prim(G_D)$. 
\item\label{ax+b_item6} The set $\widehat{G_D}\setminus\Hom(G_D,\TT)$ is a  nonempty quasi-compact open subset of $\widehat{G_D}$.
\item\label{ax+b_item8} There exist nonzero self-adjoint idempotent elements of $C^*(G_D)$. 
\item\label{ax+b_item1.5} The $C^*$-algebra $C^*(G_D)$ is not stably finite.
\end{enumerate}

\begin{proof}[Proof of claim]
Assertions \eqref{ax+b_item1} -- \eqref{ax+b_item8} are equivalent by  \cite[Ex.~4.8]{BB21}. 
A more general version of the implication  
\eqref{ax+b_item3}$\implies$\eqref{ax+b_item1} is given in Corollary~\ref{NC_cor1} below. 

The implication \eqref{ax+b_item1.5} $\implies$ \eqref{ax+b_item2} is clear.

It remains to prove  \eqref{ax+b_item1} $\implies$ \eqref{ax+b_item1.5}.
Assume that the condition in \eqref{ax+b_item1} holds.  
If 
$$\alpha^*\colon \Cc_0(\Vc^*) \times \RR \to \Cc_0(\Vc^*), \quad \alpha^*(f, t) = f\circ \ee^{t D^*}, $$   
 then $C^*(G_D) \simeq \Cc_0(\Vc^*) \rtimes_{\alpha^*} \RR$.

Let $\overline{\Vc^*}$ be the one-point compactification of $\Vc^*$ and extend $\alpha^*$ to 
$\Vc^*$ by  $\alpha^*_t(\infty)=\infty$ for every $t\in \RR$.
Then it follows from \cite[Prop.~2.14]{BB18b} and \cite[Prop.~4.6]{Pi99} 
that   the $C^*$-algebra $\Cc_0(\overline{\Vc^*})\rtimes \RR$ is not stably finite whenever 
\eqref{ax+b_item1} is true.

We now use the same argument as in \cite[Lemma~2.10]{BB18b}:
The split exact sequence $0\to \Cc_0(\Vc^*) \to \Cc(\overline{\Vc^*})\to\CC\1\to 0$ 
leads
%  by \cite[Lemma~2.82]{Ph87}
 to the split exact sequence 
	$$0\to  \Cc_0(\Vc^*)  \rtimes \RR   \to  \Cc(\overline{\Vc^*})\rtimes \RR   \to C^*(\RR)\to 0.$$ 
	Then if we  assume that  $\Cc_0(\Vc^*)  \rtimes \RR$ is stably finite,  since
	the $C^*$-algebra $C^*(\RR)$ is stably finite, 
	it follows  by Corollary~\ref{Sp88_Lemma1.5}\eqref{Sp88_Lemma1.5_item3}
	%\cite[Lemma~1.5]{Sp88} 
	that $\Cc(\overline{\Vc^*})\rtimes \RR$  is stably finite.
	This is a contradiction, hence  $\Cc_0(\Vc^*)  \rtimes \RR$ is not stably finite.
	\end{proof}
\end{example}

\subsubsection{Application of continuous fields of nilpotent Lie groups}

Let $(\ng, [\cdot, \cdot])$ be a nilpotent Lie algebra and let $\varphi\colon (0, 1]\to \GL(\ng)$, $h\mapsto\varphi_h$ be a continuous map. 
Assume  the following conditions hold:
\begin{enumerate}
\item $\varphi_1=\id$, 
\item the limit $[x, y]_0:=
\lim_{h\to 0 }[x,y]_h
%\varphi_h^{-1}([\varphi_h (x),\varphi_h(y)])
$ exists for every $x, y\in \ng$,  
\end{enumerate}
%Then we  define
%In the above condition 
where we use the bilinear map $[\cdot,\cdot]_h\colon\ng \times \ng \to \RR$ defined by 
\begin{equation}\label{defm}
[x, y]_h :=\varphi_h^{-1}([\varphi_h(x), \varphi_h(y)])
\end{equation}
for every $h\in(0,1]$.

\begin{remark}\label{defm-rem}
\normalfont
 \begin{enumerate}[\rm (i)]
 \item\label{defm-rem_i}  For all $h\in [0, 1]$, $[\cdot, \cdot]_h$ is a nilpotent Lie bracket on the vector space underlying $\ng$, and we denote by $\ast_h$ the corresponding Baker-Campbell-Hausdorff multiplication, and the corresponding connected and simply connected Lie group by
$N_h= (\ng, \ast_h)$.
 
 \item \label{defm-rem_ii} 
For every $h\in (0, 1]$, $\varphi_h \colon (\ng, [\cdot, \cdot]_h) \to (\ng, [\cdot, \cdot])$ is a Lie algebra isomorphism. 

\item\label{defm-rem_iii} For every $h\in (0, 1]$ we have
$$ \ad_h x= \varphi_h^{-1} \circ \ad (\varphi_h(x))  \circ \varphi_h$$
where  
$$ \begin{aligned} 
\ad\, x\colon \ng \to \ng, &\quad  (\ad\, x)(y)= [x, y] =[x, y]_1, \\
\ad_h x\colon \ng \to \ng, &\quad  (\ad_h x)(y)= [x, y]_h .
\end{aligned}
$$
\end{enumerate}
\end{remark}

We consider the map 
\begin{equation}\label{defm-mult}
m \colon [0, 1] \times \ng \times \ng \to \ng, \; \; m(h, x, y) =(h, x\ast_h  y).
\end{equation}
Then $m$ is continuous, by the assumptions above.  
Consider  the groupoid with equal source and target maps
$$ 
\begin{gathered} 
\Tc:=[0, 1] \times \ng  \stackrel{p}{\rightarrow} S:=[0, 1],  \; p(h, x)=x, \\
(h, x) \cdot (h, y) := (h, m_h(x, y)) = (h, x \ast_h y) \;  \text{for all } \;
(h, x), (h, y) \in \Tc_h:=p^{-1}(h).
\end{gathered}
 $$
Hence $p$ is a group bundle (depending on the map $\varphi$) with Haar system given by the Lebesgue measure. 
If follows by  \cite[Lemma~3.3]{BB18a} that $C^*(\Tc)$ is a $\Cc(S)$ -algebra that 
is $\Cc(S)$-linearly $*$-isomorphic   
to the algebra of sections of an upper semi-continuous $C^*$-bundle over~$ S$ 
whose fibre over any $s\in S$ is $C^*(\Tc_h)\simeq C^*(N_h)$.

\begin{theorem}\label{prop-cf4} 
For $\ng$ a nilpotent Lie algebra and
 $D\in \Der(\ng)$, define
the semi-direct product
$G :=N \rtimes_{\alpha_D} \RR$.
If
there exists $\epsilon \in \{-1, 1\}$ such that 
$\epsilon \Re\, z> 0$ for all $z\in \sigma(D)$, 
then 
$C^*(G)$ is  not stably finite.
\end{theorem}

\begin{proof}
Let $\Vc$ be the  underlying real vector space of the Lie algebra $\ng$, and 
denote 
$G_0:= \Vc\rtimes_{\alpha_D}\RR$.
If we proved that
 \begin{equation}\label{prop-cf4-1}
 C^*(G) \; \text{stably finite } \; \Rightarrow \; C^*(G_0) \; \text{stably finite,}
\end{equation}
then the statement follows from  Example~\ref{ax+b}.

Thus, it remains to prove \eqref{prop-cf4-1}.

For every $h \in [0, 1]$,  let $\varphi\colon [0, 1]\to \GL(\ng)$ be the map $\varphi_h(x) = h x$, 
for every $x\in \ng$. 
Consider  the deformed nilpotent Lie algebra 
$ \ng_h = (\Vc,   h [\cdot, \cdot]_\ng)$ and the corresponding nilpotent Lie group
$(N_ h , \ast_{h})$, as above. Then $N_1=N$  and $N_0= \Vc$. 
We define $G_h:= N_h \rtimes_{\alpha_D} \RR$, with the multiplication 
$ (x, t)\cdot_h  (y, s) = (x\ast_h  \ee^{t D}y, t+s)$.
Then $\Gc:= \bigsqcup\limits_{ h \in [0, 1]} G_h$ is a smooth bundle of Lie groups over $[0, 1]$. 
It follows by \cite[\S 3]{BB18a} that $C^*(\Gc)= \bigsqcup\limits_{h \in [0, 1]} C^*(G_h)$ is an upper semi-continuous bundle of $C^*$-algebras. 

On the other hand, the action $\alpha_D\colon \RR \to \Aut(N_{h})$ is independent of $h\in [0, 1]$, and thus we may choose the Haar measure on $G_h$ to be independent of $h$ as well.
Hence, by \cite[Def.~3.3, Thm.~3.5]{Ri89}, $(C^*(G_h))_{h \in [0, 1]}$ is a continuous field of $C^*$-algebras, which  is clearly  trivial away from $0$. 
The result then follows by Proposition~\ref{prop-cf2}.
\end{proof}

\subsubsection{Exponential Lie groups with exact symplectic Lie algebras and  nilradicals of codimension 1}

\begin{definition}\label{exactsympl}
\normalfont 
A solvable Lie algebra $\gg$ is said to be \emph{exact symplectic} 
if there is $\xi_0\in \gg^*$ with $\gg(\xi_0)=\{0\}$. 
Equivalently,  if $G$ is a solvable Lie group whose Lie algebra is~$\gg$, then the coadjoint orbit of $\xi_0$ is an open subset of $\gg^\ast$.  
\end{definition}

The reason for this terminology is that any Lie group $G$ as above admits a left-invariant exact symplectic form. (They are elsewhere called as Frobenius Lie groups).

\begin{lemma}\label{extra}
Let $\gg$ be a 
solvable Lie algebra  
with its  nilradical~$\ng$ with $\dim(\gg/\ng)=1$. 
Let $\zg$ be the centre of $\ng$. 
Let $G$  be a connected simply connected Lie group with its Lie algebra $\gg$ and $N\subseteq G$ be the connected subgroup corresponding to the subalgebra $\ng\subseteq\gg$.
Then the following assertions are equivalent: 
\begin{enumerate}[\rm(i)]
\item\label{extra-i} $\gg$ is exact symplectic.
\item\label{extra-i-2} $\gg$ is exact symplectic
and has  
%exactly 
two open coadjoint orbits.
\item\label{extra2-i} There is at least an open point in $\widehat{G}$. 
\item\label{extra2-i-2} There  are 
%exactly  
two open points in $\widehat{G}$. 
\item \label{extra-ii}
%\begin{itemize}
$[\gg, \zg]\ne \{0\}$,  $\dim \zg =1$ and the nilpotent Lie group $N=\exp \ng$ has generic flat coadjoint orbits
{\rm (}or equivalently, there is $\ell \in \ng^*$ such that $\ng(\ell) =\zg$.\rm{)}
%\end{itemize}
\end{enumerate}

\end{lemma}

\begin{proof}
The equivalences  \eqref{extra2-i} $ \iff$ \eqref{extra2-i-2}  $\iff$ \eqref{extra-ii}  follows  from \cite[Thm.~4.5]{KT96}.

The implication  \eqref{extra-i-2} $\Rightarrow$ \eqref{extra-i} is trivial. 
It remains to prove  \eqref{extra-i} $\Rightarrow$   \eqref{extra-ii} $\Rightarrow$ \eqref{extra-i-2}. 
 For every $\xi\in\gg^*$ we denote by $\Oc_\xi\subseteq\gg^*$ its corresponding coadjoint orbit. 

\eqref{extra-ii} $\Rightarrow$\eqref{extra-i-2}: 
We prove by contradiction the following assertion: 
\begin{equation}\label{extra_proof_eq0}
\text{if }\xi\in\gg^*\text{ and }\ng(\xi\vert_\ng)=\zg\text{ then }
\gg(\xi)=\{0\}. 
\end{equation}
Hence let us assume $\gg(\xi)\ne\{0\}$ and let us denote $\ell:=\xi\vert_\ng\in\ng^*$. 
The hypothesis~\eqref{extra-ii} implies that that $\dim\ng$ is an odd integer and $\dim\gg$ is an even integer. 
Since $\dim\Oc_\xi=\dim\gg/\gg(\xi)$ and this is an even integer, it follows that $\dim\gg(\xi)$ is also an even integer, hence $\dim\gg(\xi)\ge2$. 
Then $\dim\ng+\dim\gg(\xi)>\dim\gg$, hence $\ng\cap\gg(\xi)\ne\{0\}$. 
On the other hand 
$$\ng\cap\gg(\xi)=\{X\in\ng\mid\langle\xi,[X,\gg]\rangle=\{0\}\}
\subseteq\ng(\xi\vert_\ng)=\ng(\ell)=\zg$$
hence, since $\dim\zg=1$, we obtain $\ng\cap\gg(\xi)=\zg$. 
In particular $\zg\subseteq\gg(\xi)$. 

Now, selecting any $Y\in\gg\setminus\ng$,  
the centre $\zg\subseteq\ng$ is invariant to the derivation $(\ad_\gg Y)\vert_\ng\in\Der(\ng)$. 
%Specifically, for any $X_0\in\zg\setminus\{0\}$ 
%and all $X\in\ng$ we have 
%$$0=[Y,[X,X_0]]=[X,[Y,X_0]]+[[Y,X],X_0]]=[X,[Y,X_0]],$$
%where we used that $[X,X_0]=[[Y,X],X_0]]=0$ since  $X,[Y,X]\in\ng$. 
It then follows that $[Y,X_0]\in\zg$ for any $X_0\in\zg\setminus\{0\}$. 
Since $\gg=\ng\dotplus\RR Y$ and $[\ng,\zg]=\{0\}$, the hypothesis $[\gg,\zg]\ne\{0\}$ 
implies $[Y,X_0]\ne0$, hence there exists $a\in\RR^\times$ with $[Y,X_0]=aX_0$. 
Thus $\langle\xi,[Y,X_0]\rangle=\langle\xi,aX_0\rangle
=a\langle\ell,X_0\rangle\ne0$, 
using the assumption $\ng(\ell)=\zg=\RR X_0$. 
This shows that $X_0\not\in\gg(\xi)$, which is a contradiction 
with the above conclusion $\zg\subseteq\gg(\xi)$. 
This completes the proof of~\eqref{extra_proof_eq0}, which shows that $\gg$ has open coadjoint orbits. 

In order to prove that there are 
%exactly 
two open coadjoint orbits, 
we consider the set of generic points 
$$\gg^*_{\rm gen}:=\{\xi\in\gg^*\mid\gg(\xi)=\{0\}\}.$$ 
The set $\gg^*_{\rm gen}$ is the union of the open coadjoint orbits of $\gg$, which are connected and mutually disjoint, 
hence are also relatively closed in  $\gg^*_{\rm gen}$. 
Thus the open coadjoint orbits are the connected components of $\gg^*_{\rm gen}$. 
On the other hand, by \eqref{extra_proof_eq0}, we have 
$$\gg^*_{\rm gen}=\{\xi\in\gg^*\mid \xi\vert_{\zg}\in\zg^*\setminus\{0\}\}=\gg^*\setminus\zg^\perp$$
since for any $\ell\in\ng^*$ the equality $\ng(\ell)=\zg$ is equivalent to $\ell\vert_\zg\in\zg^*\setminus\{0\}$, 
by the hypothesis on~$\ng$ and $\zg$. 
Thus $\gg^*_{\rm gen}$ is the complement of a hyperplane in $\gg^*$, 
and then $\gg^*_{\rm gen}$ is the union of two open half-spaces in 
 $\gg^*$. 
The above remarks then show that these open half-spaces are just 
 the open coadjoint orbits of $\gg^*$, hence there are 
 %exactly 
 two such orbits.

\eqref{extra-i}$\Rightarrow$\eqref{extra-ii}: 
Fix $\xi_0 \in \gg^*$ with  $\gg(\xi_0)=\{0\}$, which exists by hypothesis. 
Assume that
$[\gg, \zg] =\{ 0\}$; then $\zg \subseteq \gg(\xi_0)$, and  $\dim \zg \ge 1$ since $\ng$ is nilpotent.  
Thus $\dim \gg(\xi_0)\ge 1$ and 
this is a contradiction with \eqref{extra-i}, therefore
$[\gg, \zg]\ne \{0\}$. 

For $\xi_0\in \gg^*$ as above, define the bilinear functional  
$$B_{\xi_0} \colon \gg\times \gg\to \RR, 
\quad B_{\xi_0}(X, Y) = \langle \xi_0, [X, Y]\rangle.$$ 
We have that $\zg \subseteq \ng^{\perp_{B_{\xi_0}}}$, therefore $\dim \ng^{\perp_{B_{\xi_0}}}\ge 1$.
If there exists $Y_0\in \ng^{\perp_{B_{\xi_0}}}\setminus \ng$ then, since $\dim(\gg/\ng)=1$, 
$\gg =\ng \dotplus \RR Y_0$. 
On the other hand, $Y_0 \perp_{B_{\xi_0}} \ng$, hence $Y_0\perp_{B_{\xi_0}}\zg$, while 
$\ng\perp_{B_{\xi_0}}\zg$;  it follows that $\gg \perp_{B_{\xi_0}}\zg$. 
Since $\zg \ne \{0\}$ this is a contradiction with the fact that $B_{\xi_0}$ is symplectic. 
Thus $\ng^{\perp_{B_{\xi_0}}}\subseteq \ng$.
For arbitrary $X\in \gg \setminus \ng$, $\dim (X^{\perp_{B_{\xi_0}}}) =\dim \gg -1$.
Hence 
if $\dim(\ng^{\perp_{B_{\xi_0}}}) \ge 2$, then $\{0\}\ne X^{\perp_{B_{\xi_0}}}\cap \ng^{\perp_{B_{\xi_0}}}
=(\RR X+\ng)^{\perp_{B_{\xi_0}}}
=\gg^{\perp_{B_{\xi_0}}}$, which is again a contradiction.
It follows that $\dim(\ng^{\perp_{B_{\xi_0}}})=1$, hence
 $\ng (\xi_0\vert_{\ng}) =\zg$ and $\dim \zg =1$, and we are done. 
%This completes the proof. 
 \end{proof}

The Lie groups $G$ in Definition~\ref{exactsympl} are not necessarily exponential groups. 
See for instance~\cite[\S 6]{BB18a}. 
In the special case of exponential Lie groups we obtain the following result. 

\begin{corollary}\label{extra-cor}
Let $G$ be an exponential Lie group with its Lie algebra $\gg$ and nilradical $\ng$ with $\dim (\gg/\ng) =1$.
Then the following assertions are equivalent. 
\begin{enumerate}[\rm(i)]
\item \label{cor_extra2-0} There is an open point in $\Prim(G)$.
\item\label{cor_extra2-i} $\gg$ is exact symplectic.
\item \label{cor_extra-ii}
\begin{itemize}
\item{}  There is a continuous action $\alpha \colon \RR \to \Aut N$ such that $G = N \rtimes_\alpha \RR$ 
and $\alpha$ acts non-trivially on the centre $Z$ of $N$.
\item{} The nilpotent Lie group $N$ has generic flat coadjoint orbits and its centre is 1-dimensional.
\end{itemize}
\end{enumerate}
\end{corollary}

\begin{proof}
The assertion is a consequence of Lemma~\ref{extra} and Remark~\ref{P3_group}.
\end{proof}

\subsubsection{More on  exact symplectic groups $G= N\rtimes \RR$}
Let $N$ be a nilpotent Lie group and $Z$ the centre of $N$. 
%Let $\ng$ and $\zg$ be the Lie algebras of $N$ and $Z$, respectively.
We assume $\dim Z=1$ and we define the ideal 
$$ \Ic:= \bigcap_{\sigma \in \widehat{N/Z}}  \Ker_{C^*(N)}(\sigma)\subseteq C^*(N). $$
If $\Psi \colon C^*(N) \to C^{*}(N/Z)$ is the surjective morphism given by 
$$ (\Psi(f))(xZ) = \int_Z f(xz) \de z,$$
then, by  \cite[Prop.~8.C.8]{BkHa20}, we have the short exact sequence 
\begin{equation}\label{cf-0}
 0 \longrightarrow  \Ic \longrightarrow C^*(N) \stackrel{\Psi}{\longrightarrow}  C^*(N/Z) \longrightarrow  0.
 \end{equation}

Now assume that the group $N$ has  generic flat coadjoint orbits, and denote $\dim N = 2d +1$, $d \in \NN$.
Let $\alpha \colon \RR\to \Aut(N)$  be a continuous action that acts non-trivially on $Z$, that is, 
$\alpha_t\vert_Z\ne \id_Z$ for some, hence all, $t \in \RR\setminus \{0\}$. 
Then there is $\tau_0 \in \RR\setminus \{0\}$ such that $\de \alpha_t\vert_{\zg} = \ee^{\tau_0 t} \id_\zg$.

Fix $X_0\in \zg \setminus \{0\}$ and let $\pi_1\colon N \to \Bc(L^2(\RR^d))$ be the unitary  irreducible representation such that 
$$
\pi_1(\exp_N (s X_0))=\ee^{\ie s} \id_{L^2(\RR^d)}. 
$$
Then 
\begin{equation}\label{cf-4}
(\pi_1 \circ \alpha_t)(\exp_N(sX_0))= \pi_1(\exp_N(s\ee^{\tau_0t}X_0))= 
\ee^{\ie s\ee^{\tau_0t}} \id_{L^2(\RR^d)}.
\end{equation}

Let $C\colon L^2(\RR^d) \to L^2(\RR^d)$ be the usual complex  conjugation $v \mapsto \overline{v(\cdot)}$. 
For every $t\in \RR$ define 
\begin{equation}\label{cf-4.5} 
\pi_t\colon N \to \Bc(L^2(\RR^d)),
 \quad \pi_t:= 
 \begin{cases} \pi_1\circ \alpha_{\log t} & \text{if } t >0, \\
 C \pi_{-t} C & \text{if } t < 0.
 \end{cases}
 \end{equation}
 Then \eqref{cf-5} and \eqref{cf-4} give 
 $$
 \pi_t (\exp_N (sX_0))= \begin{cases} \ee^{\ie  t^{\tau_0} s} \id_{L^2(\RR^d)}  & \text{if } t >0, \\
   \ee^{-\ie  \vert t\vert^{\tau_0} s} \id_{L^2(\RR^d)}  & \text{if } t < 0.
 \end{cases}
$$
Hence the map 
$$ \RR^{\times} \to \what{N}\setminus \what{N/Z}, \quad t \mapsto [\pi_t]$$
is a homeomorphism. 

We have thus obtained that there is a $*$-isomorphism 
\begin{equation}\label{cf-5}
 \Phi\colon \Ic \to \Cc_0(\RR^{\times}, \Kc(L^2(\RR^d))),  \;\;
 a\mapsto (s\mapsto \pi_s(a)=\Phi(a)(s))
 \end{equation}
 that satisfies 
 \begin{equation}\label{cf-6}
 ( \Phi(\overline{a}))(t)= \pi_t(\overline{a}) = C\overline{\pi_t}(a) C = C\pi_{-t}(a)  C,\; \text{
 for all } t\in \RR^{\times}.
 \end{equation}
(Compare with the proof of Lemma~\ref{rcrossed}.)

Denote 
$$ \Ic_{\pm}:=\Ic \cap \bigcap\limits_{t>0} \Ker_{C^*(N)} \pi_{\mp t}, $$
which are ideals of $\Ic$.
Then by \eqref{cf-5}, \eqref{cf-6} we get 
\begin{gather}
\Ic = \Ic_{+} \dotplus \Ic_{-}, \label{cf-7}\\
\overline{\Ic_{+}}=\Ic_{-},  \label{cf-8}\\
\alpha_t(\Ic_{\pm})=\Ic_{\pm}, \quad \text{for all } \, t \in\RR.\nonumber
\end{gather}

\begin{lemma}\label{cf-lemma6}
With the notation and in the conditions above, 
$$ \Ic_{+}\rtimes_{\alpha} \RR \simeq \Kc(L^2(\RR^d))\otimes \Kc(L^2(\RR^d)).$$
\end{lemma}

\begin{proof}
The restriction of $\Phi$ to $\Ic_{+}$ gives a $*$-isomorphism 
$$\Phi\vert_{\Ic_+}\colon \Ic_{+} \, \widetilde{\longrightarrow}\,  \Cc_0(\RR^\times_+, \Kc(L^2(\RR^d))).$$ 
By \eqref{cf-6} and \eqref{cf-4.5}, for every $t\in \RR$ and $s>0$ we have
%\begin{equation}\label{-cf-10}
\begin{align*}
\Phi(\alpha_t(a))(s)  & = \pi_s(\alpha_t(a)) 
  =(\pi_1\circ\alpha_{\log s})(\alpha_t(a))
  =\pi_1(\alpha_{\log (s \ee^t)}(a)) 
= \pi_{s\ee^t}(a) \\
&=\Phi(a)(s\ee^t).
\end{align*}
%\end{equation}
On the other hand, if we define
$ \rho\colon \RR_+^\times \times \Cc_0(\RR^\times_+) \to \Cc_0(\RR^\times_+)$, 
$(\rho_t (f))(s)= f(st)$, 
then
\begin{equation}\label{cf-11}
\Cc_0(\RR^\times_+) \rtimes_\rho \RR^\times_+ \simeq \Kc(L^2(\RR^\times_+)).
\end{equation}
Then the assertion in the statement follows from the commutative diagram
$$ \xymatrix{
\RR \times \Ic_+ \ar[r]^{\alpha} \ar[d]_{\exp \times \Phi} & \Ic_+\ar[d]^{\Phi}\\
\RR_+^\times  \times (\Cc_0(\RR^\times_+)\otimes \Kc(L^2(\RR^d)))\ar[r]^{\;\;  \rho\otimes \id_\Kc}
 & \Cc_0(\RR^\times_+)\otimes \Kc(L^2(\RR^d))
}
$$
and \eqref{cf-11}. 
\end{proof}

Denote 
\begin{equation}\label{cf-Jc}
\Jc: = \Ic_+\rtimes_\alpha \RR.\end{equation}
Then
$\Jc$ is an elementary $C^*$-algebra, by Lemma~\ref{cf-lemma6}. 

\begin{theorem}\label{cf-prop7}
Let $G$ is a solvable Lie group  with exact symplectic Lie algebra $\gg$  of dimension $2d+2$. 
Assume that the nilradical $N$  of $G$  is of codimension 1, and 
let $Z$ be the centre of $N$.
\begin{enumerate}[\rm (i)]
\item\label{cf-prop7_i} There is an ideal $\Jc$ of $C^*(G)$ 
 such that 
$\Jc \simeq \Kc(L^2(\RR^d)) \otimes \Kc(L^2(\RR^d))$ and
there is  the following short exact sequence
\begin{equation}\label{cf-prop7-eq1} 
0\longrightarrow \Jc \dotplus \overline{\Jc} \stackrel{\iota}{\longrightarrow} C^*(G ) \stackrel{\psi}{\longrightarrow} 
C^*(G/Z) \rightarrow 0
\end{equation}
\item\label{cf-prop7_ii} $K_0(\Jc\dotplus \overline{\Jc})^+ \cap \Ker K_0(\iota) =\{0\}$ if and only if $\dim(G) \in 4 \ZZ$.
\end{enumerate}
\end{theorem}

\begin{proof}
By Lemma~\ref{extra}, we have that $\dim Z=1$, $\gg/\ng\simeq \RR$, the  continuous action
$\alpha\colon \RR \to \Aut(N)$  is   non-trivial on the centre $Z$ of $N$ and 
$G= N\rtimes \RR$.

It follows from \eqref{cf-0} 
that we have the short exact sequence
$$ 0\longrightarrow \Ic \rtimes \RR \stackrel{\iota}\longrightarrow  C^*(N \rtimes \RR)  \longrightarrow 
C^*((N/Z) \rtimes \RR) \longrightarrow 0,
$$
where $\iota$ is the inclusion map.
Then Assertion \eqref{cf-prop7_i} is a consequence of  \eqref{cf-Jc}, \eqref{cf-7}, \eqref{cf-8} and of Lemma~\ref{cf-lemma6}, 
using the Lie group isomorphism $(N\rtimes\RR)/Z\simeq(N/Z)\rtimes\RR$.

\eqref{cf-prop7_ii}
First note that, $\dim G= \dim(N\rtimes \RR)=2d+2$, hence either 
$\dim G\in 4\ZZ +2 $  or $\dim G\in 4\ZZ$. 
Then in the six-term exact sequence corresponding to \eqref{cf-prop7-eq1}
$$ \xymatrix{
K_0(\Jc \dotplus  \overline{\Jc})  \ar[r]^{K_0(\iota)}
 & K_0 (C^*(N\rtimes \RR))\ar[r]^{K_0(\psi)} 
&  K_0 (C^*((N/Z)\rtimes \RR))\ar[d]\\
K_1 (C^*((N/Z)\rtimes \RR)) \ar[u] & \ar[l]^{K_1(\psi)}
K_1 (C^*(N\rtimes \RR)) &\ar[l]^{K_1(\iota)}
K_1(\Jc\dotplus \overline{\Jc})
}
$$
we have 
$$ K_0 (C^*((N/Z)\rtimes \RR))=K_1(\Jc\dotplus \Jc) = K_1 (C^*(N\rtimes \RR))=\{0\}$$
and 
$$ K_0(\Jc \dotplus \overline{\Jc})\simeq \ZZ \times \ZZ, \; \; K_0 (C^*(N\rtimes \RR))\simeq K_1 (C^*((N/Z)\rtimes \RR))\simeq \ZZ.$$
More specifically, there is an isomorphism
$$ 
\chi\colon K_0(\Jc) \times K_0(\overline{\Jc}) \,  \wtilde{\longrightarrow} \, 
K_0(\Jc \dotplus \overline{\Jc}),
$$ 
such that  $p$ is a minimal projection in $\Jc$, we have
$$K_0(\Jc) \times K_0(\overline{\Jc})  =\{ (m[p]_{0, \Jc}, 
n [\overline{p}]_{0, \overline{\Jc}})\mid m, n \in \ZZ \}\simeq \ZZ \times \ZZ,  $$
and $\chi (m[p]_{0, \Jc}, n[\overline{p}]_{0, \overline{\Jc}})
 = [mp+n\overline{p}]_{0, \Jc\dotplus\overline{\Jc}}
$ if $m,n\in\{0,1\}$.
Thus
\begin{equation}\label{cf-prop7-eq2}
( K_0(\iota) \circ  \chi ) (m[p]_{0, \Jc}, n[\overline{p}]_{0, \overline{\Jc}})= 
m[p]_{0} + n[\overline{p}]_0 
\text{ for all }m,n\in\ZZ.
\end{equation} 
\textit{Case 1: $\dim (N\rtimes \RR)\in 4 \ZZ +2$.} 

In this case $[p]_{0} =- [\overline{p}]_0$, hence 
 $$ [p+\overline{p}]_{0, \Jc\dotplus\overline{\Jc}}\in
 ( K_0(\Jc \dotplus \overline{\Jc})^+ \cap \Ker K_0(\iota) ) \setminus \{0\}.$$
\textit{Case 2: $\dim (N\rtimes \RR)\in 4 \ZZ$.} 

In this case $[p]_{0} =[\overline{p}]_0$, hence   the morphism
$ K_0(\iota) \circ  \chi\colon \ZZ \times \ZZ \to K_0(C^*(N\rtimes\RR) ) \simeq \ZZ$, 
is given by 
$$ (K_0(\iota) \circ  \chi ) (m[p]_{0, \Jc}, n[\overline{p}]_{0, \overline{\Jc}})= 
(m+n) [p]_{0}.$$
Thus 
$ K_0(\Jc \dotplus \overline{\Jc})^+ \cap \Ker K_0(\iota) = \{0\}$, 
and this finishes the proof.
\end{proof}

\begin{corollary}\label{cf-cor8}
Let $G$ be a solvable Lie group  with exact symplectic Lie algebra~$\gg$. 
Assume that the nilradical $N$  of $G$  is of codimension 1 and let 
$Z$ is the centre of~$N$.
Then the following assertions hold true:
\begin{enumerate}[\rm (i)]
\item\label{cf-cor7-i} $C^*(G)$ is stably finite 
if and only if $C^*(G/Z)$ is. 
\item\label{cf-cor7-ii} $C^*(G)$ is AF-embeddable
if 
%and only if 
$C^*(G/Z)$ is.
\end{enumerate}
\end{corollary}

\begin{proof} By Lemma~\ref{extra}   it is enough to prove the corollary for groups of the form $G= N \rtimes \RR$ 
and $G/Z = (N/Z) \rtimes \RR$ such that the nilradical $N$ of $G$  and the action of $\RR$ by automorphisms of $N$ satisfies the hypotheses of Theorem~\ref{cf-prop7}.

\eqref{cf-cor7-i} ``$\Leftarrow$'' The assertion follows from  Theorem~\ref{cf-prop7}
and \cite[Lemma~1.5]{Sp88}.

 ``$\Rightarrow$'' Assume that $C^*((N/Z)\rtimes \RR)$ is not  stably finite. 
 Then, by Proposition~\ref{P1}, there is $0\ne q _0\in \Pc(M_k \otimes C^*((N/Z)\rtimes \RR)$. 
(Note that $\dim ((N/Z) \rtimes \RR)$ is odd, so that $K_0( C^*((N/Z)\rtimes \RR ))=0$; hence
 $[q]_{0, C^*((N/Z)\rtimes \RR)}=0$ anyway.) 
 Then,  by \cite[Thm.~1]{Ch83}, there exists $0\ne q \in  \Pc(M_k \otimes C^*( N\rtimes \RR))$ with 
 $(1\otimes \psi)(q) =q_0$; hence in particular,  
 $q\not \in M_k \otimes (\Jc \dotplus \overline{\Jc})$.
The result then follows from Proposition~\ref{proj}.

\eqref{cf-cor7-ii} 
Use Theorem~\ref{cf-prop7}\eqref{cf-prop7_i}, the above Assertion~\eqref{cf-cor7-i}, and  \cite[Thm.~1.15]{Sp88}.
\end{proof}

\section{Examples}
\label{Sect4}

In this final section we illustrate the results from Section~\ref{section4k},  effectively using them for establishing existence or lack of various finite approximation properties for the $C^*$-algebras of several concrete solvable Lie groups. 
For instance, we study exponential solvable Lie groups that have exactly two open coadjoint orbits, and whose nilradical is either a Heisenberg group (Proposition~\ref{Heis}), or the free 2-step nilpotent Lie group with 3 generators (Theorem~\ref{N6N15}), or a Heisenberg-like group associated to a finite-dimensional real division algebra (Theorem~\ref{N6N17}). 

\subsection{Some semidirect products} 
In Lemma~\ref{NC_lemma} and Proposition~\ref{NC} below we use some notation related to quasi-orbits of group actions (see \cite[\S 2.1]{BB21}). 
For any group action $\Gamma \times X \to X$ on the topological space $X$, we denote by 
$$ (X/\Gamma)^\approx :=\{\overline{\Gamma x}\mid x\in X\}$$
the set of all orbit closures, regarded as a topological subspace of the space $\text{Cl}(X)$ of all closed subsets of $X$, endowed with the upper topology. 
The map $X\to  (X/\Gamma)^\approx$, $x\mapsto \overline{\Gamma x}$, 
is called the quasi-orbit map.

\begin{lemma}\label{NC_lemma}
	Let $\Vc$ be a finite-dimensional real vector space and $T\in\End(\Vc)$ with $\sigma(T)\cap\ie\RR\subseteq\{0\}$, for which there exist $w_1,w_2\in\sigma(T)$ with $\Re\, w_1\le 0\le\Re\, w_2$. 
	If we define the abelian group $\exp(\RR T):=\{\ee^{sT}\mid s\in\RR\}\subseteq\End(\Vc)$ with its natural action on~$\Vc$, then there exists a continuous open mapping $\Psi\colon(\Vc/\exp(\RR T))^\approx\to\RR$. 
\end{lemma}

\begin{proof}
	Case 1: $0\in\sigma(T)$, that is, $\Ker T\ne\{0\}$. 
	
	Then, using the Jordan decomposition, we can obtain a linear subspace $\Vc_0\subsetneqq\Vc$ with $T(\Vc)\subseteq\Vc_0$. 
	In particular $T(\Vc_0)\subseteq\Vc_0$, and this implies that 	
	the group $\exp(\RR T)$ naturally acts on $\Vc/\Vc_0$, and the quotient map $q\colon\Vc\to\Vc/\Vc_0$ is $\exp(\RR T)$-equivariant. 
	We then obtain the commutative diagram 
	$$\xymatrix{
		\Vc \ar[r]^q \ar[d] & \Vc/\Vc_0 \ar[d] \\
		(\Vc/\exp(\RR T))^\approx \ar[r]^{q^\approx}& ((\Vc/\Vc_0)/\exp(\RR T))^\approx
	}$$
	whose vertical arrows are quasi-orbit maps, hence they are continuous and open. 
	(See for instance \cite[Lemma 2.3]{BB21}.) 
	Since $q$ is also continuous and open, it then directly follows that $q^\approx$ is continuous and open. 
	
	Recalling that $T(\Vc)\subseteq\Vc_0$, the action of $\exp(\RR T)$ on $\Vc/\Vc_0$ is trivial, hence the right-most arrow in the the above diagram is actually a homeomorphism and the composition of its inverse with $q^\approx$ is a continuous open map 
	$$\Psi_1\colon(\Vc/\exp(\RR T))^\approx\to\Vc/\Vc_0.$$
	The real vector space $\Vc/\Vc_0$ is different from $\{0\}$, hence there exists a non-zero linear functional $\xi\colon\Vc/\Vc_0\to\RR$, 
	and then the mapping $\Psi:=\xi\circ\Psi_1$ has the required properties. 
	
	Case 2: $0\not \in\sigma(T)$. 
	
	Since $\sigma(T)\cap\ie\RR\subseteq\{0\}$, we then have the direct sum decomposition $\Vc=\Vc_+\dotplus\Vc_-$, 
	where $\Vc_\pm$ is the direct sum of the real generalized eigenspaces corresponding to all $w\in\sigma(T)$ with $\pm\Re\,w>0$. 
	(See \cite[Sect. 2]{BB18b}.) 
	Denoting $n_\pm:=\dim\Vc_\pm$, we obtain $n_-n_+\ne0$ by hypothesis. 
	Let $n:=n_++n_-=\dim\Vc$ and define 
	$$T_0:=\begin{pmatrix}
	\1_{n_+} &  0 \\
	0 & -\1_{n_-}
	\end{pmatrix}
	\in M_n(\RR).$$
	It follows by \cite[Lemma 2.1]{BB18b} 
	that there exists a homeomorphism $\Theta\colon \Vc\to\RR^n$ 
	satisfying $\Theta\circ\ee^{sT}=\ee^{sT_0}\circ\Theta$ for all $s\in\RR$, 
	hence we obtain a homeomorphism 
	$$\Theta^\approx\colon(\Vc/\exp(\RR T))^\approx
	\to(\RR^n/\exp(\RR T_0))^\approx.$$
	Now define the injective linear map 
	$$p\colon \RR^n=\RR^{n_+}\times\RR^{n_-}\to\RR^2,\quad 
	((x_1,\dots,x_{n_+}),(y_1,\dots,y_{n_-}))\mapsto(x_1,y_1)$$
	(which makes sense since  $n_-n_+\ne0$)
	and let 
	$$S:=\begin{pmatrix}
	1 & \hfill 0 \\
	0 & -1
	\end{pmatrix}\in M_2(\RR).$$
	We have $p\circ T_0=S\circ p$ and $p$ is continuous and open, hence we obtain a continuous open mapping 
	$$p^\approx\colon (\RR^n/\exp(\RR T_0))^\approx\to 
	(\RR^2/\exp(\RR S))^\approx.$$
	Furthermore, we note that the mapping 
	$$\varphi\colon\RR^2\to\RR,\quad \varphi(x,y):=xy$$
	is continuous and open, since its restriction to $\RR^2\setminus\{(0,0)\}$ is actually a submersion while $\varphi((-a,a)^2)=(-a^2,a^2)$ for all $a\in(0,\infty)$. 
	On the other hand, we have $\varphi\circ\ee^{tS}=\varphi$ for all $t\in\RR$, 
	hence there exists a commutative diagram 
	$$\xymatrix{
		\RR^2 \ar[r]^\varphi \ar[d] & \RR \\
		(\RR^2/\exp(\RR S))^\approx \ar@{.>}[ur]_{\varphi^\approx}
	}
	$$
	whose vertical arrow is a quasi-orbit map, hence is continuous and open, and this directly implies that $\varphi^\approx$ is continuous and open as well. 
	Finally, the composition 
	$$\Psi:=\varphi^\approx\circ p^\approx\circ\Theta^\approx\colon 
	(\Vc/\exp(\RR T))^\approx\to\RR $$ 
	is a continuous open mapping, as required. 
\end{proof}

\begin{proposition}\label{NC}
	Let $\ng$ be a nilpotent Lie algebra with its center~$\zg$, and $D\in\Der(\ng)$  
	satisfying the conditions 
	\begin{equation}\label{NC_eq1}
	\sigma(D)\cap\ie\RR\subseteq\{0\}
	\end{equation}
	and 
	\begin{equation}\label{NC_eq2}
	\text{there exist }w_1,w_2\in\sigma(D\vert_\zg)\text{ with }
	\Re\, w_1\le 0\le\Re\, w_2.
	\end{equation}
	If $G$ is a simply connected Lie group with its Lie algebra $\gg:=\ng\rtimes\RR D$, then $G$ is an exponential solvable Lie group and there exists a continuous open mapping $\Phi\colon\Prim(G)\to\RR$.  
\end{proposition}

\begin{proof}
	The hypothesis \eqref{NC_eq1} ensures that $G$ is an exponential solvable Lie group. 
	
	Step 1: 
	If $0\in\sigma(D\vert_\zg)$, then $0\ne\Ker (D\vert_\zg)=\zg\cap\Ker D$. 
	On the other hand, it is easily seen that $\zg\cap\Ker D$ is contained in the center of $\gg$, hence it follows that  the center $Z_G$ of the  exponential solvable Lie group $G$ satisfies~$k:=\dim Z_G\ge 1$, and then the assertion follows at once using the (continuous open) restriction mapping ${\rm Res}^G_{Z^G}\colon \Prim(G)\to\widehat{Z_G}$ given by \cite[Lemma 2.11]{BB21} 
	along with the fact that $\widehat{Z_G}$ is homeomomorphic to $\RR^k$. 
	
	Hence we may assume  $0\not\in\sigma(D\vert_\zg)$ from now on, without any loss of generality. 
	
	Step 2: 
	Let $A:=\{\ee^{tD}\mid t\in\RR\}\hookrightarrow\Aut(\ng)=\Aut(N)$, 
	where $N:=(\ng,\cdot)$ is the simply connected Lie group associated with~$\ng$. 
	We regard $A$ as an abelian Lie group that is isomorphic to $(\RR,+)$ since $D\ne0$ by Step~1. 
	(See also \eqref{NC_eq1}.) 
	Also let $Z=(\zg,\cdot)=(\zg,+)$, the center of $N$. 
	
	We have the semidirect product of Lie groups $G=N\rtimes A$, hence $C^*(G)=C^*(N)\rtimes A$, which carries a natural dual action $\widehat{A}\times C^*(G)\to C^*(G)$. 
	We then obtain the composition of continuous open maps 
	$$\Phi_1\colon \Prim(G) \to(\Prim(G)/\widehat{A})^\approx
	\simeq(\Prim (N)/A)^\approx 
	\to (\widehat{Z}/A)^\approx,$$
	where the left-most map is the quasi-orbit map corresponding to the natural action $\widehat{A}\times\Prim(G)\to\Prim(G)$, 
	the middle homeomorphism is given by \cite[Cor. 2.5]{GL86}, 
	while the right-most map is obtained as in the proof of \cite[Prop. 4.7]{BB21} using the fact that the restriction mapping $R^N
	\colon \Prim(N) \to\widehat{Z}$ is not only continuous and open by \cite[Lemma 2.11]{BB21}, but also $\Aut(N)$-equivariant. 
	
	Step 3: 
	The canonical homeomorphism $E\colon \zg^*\to \widehat{Z}$, $\xi\mapsto\ee^{\ie\xi}$, intertwines the group actions 
	$$\RR\times\zg^*\to\zg^*,\quad (t,\xi)\mapsto\xi\circ\ee^{tD}$$
	and 
	$$\RR\times\widehat{Z}\to\widehat{Z},\quad (t,\chi)\mapsto\chi\circ\ee^{tD}\vert_\zg$$
	hence we obtain the homeomorphism 
	$$E^\approx\colon(\zg^*/\RR)^\approx\to(\widehat{Z}/\RR)^\approx.$$
	On the other hand, the hypothesis~\eqref{NC_eq2} shows that we may use Lemma~\ref{NC_lemma} for $T:=(D\vert_\zg)^*\in\End(\zg^*)$ to obtain a continuous open mapping $\Psi\colon (\widehat{Z}/\RR)^\approx\to\RR$, 
	and then the mapping $\Phi_2:=\Psi\circ (E^\approx)^{-1}\colon (\widehat{Z}/\RR)^\approx\to\RR$ is continuous and open. 
	Finally, using the mapping $\Phi_1$ from Step~2, we obtain the continuous open mapping $\Phi:=\Phi_2\circ\Phi_1\colon\Prim (G)\to\RR$, as required. 
\end{proof}

\begin{corollary}\label{NC_cor1}
	In Proposition~\ref{NC}, the topological space $\Prim(G)$ contains  no non\-empty quasi-compact open subsets, and the $C^*$-algebra $C^*(G)$ is AF-embeddable.  
\end{corollary}

\begin{proof}
	Any nonempty quasi-compact open subset of $\Prim(G)$ would be mapped via $\Phi$ onto a nonempty compact open subset of $\RR$, but there are no such subsets of the connected noncompact space~$\RR$. 
	Moreover, $C^*(G)$ is nuclear since $G$ is an amenable group. 
	One can then use \cite[Cor. B]{Ga20} to see that $C^*(G)$ is AF-embeddable. 
\end{proof}

\subsection{The semidirect product $H_n\rtimes \RR$}
We now give the simplest example of exponential solvable Lie group whose primitive ideal space has finite open subsets and whose $C^*$-algebra is nevertheless AF-embeddable. (See \eqref{Heis_item4} in Proposition~\ref{Heis}.)

Let $ H_n$ be the $(2n+1)$-dimensional Heisenberg group with its Lie algebra
$\hg:=\hg_n=\spa\{Z,Y_1,\dots,Y_n,X_1,\dots,X_n\}$, 
where
$$
[Y_j,X_j]=Z
$$ 
for $j=1,\dots,n$.
Denote  by $\zg: =\RR Z$ the centre of $\hg_n$. 
For any $D\in\Der(\hg_n)$ there exists $d_\zg\in\RR$ with 
$D\vert_\zg=d_\zg\id_\zg$, hence there exists an operator 
 $$\widetilde{D}\colon\hg_n/\zg\to\hg_n/\zg,  \quad V+\zg\mapsto D(V)+\zg.$$
We also define $\alpha_D\colon\RR\to\Aut(H_n)$, 
$t\mapsto\exp (tD)$. 

\begin{proposition}\label{Heis} 
Let  $D\in\Der(\hg_n)$ be such that $\sigma(D)\cap\ie\RR\subseteq\{0\}$, 
and  let $G_{n,D}:=H_n \rtimes_{\alpha_D}\RR$  be the corresponding semidirect product.
Then  we have 
\begin{enumerate}[{\rm(i)}]
\item\label{Heis_item1} If $d_\zg=0$ then  $C^*(G_{n,D})$ is AF-embeddable for every $n\ge 1$ and there is no nonempty quasi-compact  open  subset of $\Prim(G_{n,D}) $.
\item\label{Heis_item1.5}
 If $d_\zg\ne 0$, then there are  two open points  in $\Prim(G_{n,D})$, for every $n\ge 1$. 
\item\label{Heis_item2}  
If  $d_\zg\ne 0$ and there exists $\epsilon \in \{-1, 1\}$ such that 
$\epsilon \Re\, z> 0$ for all $z\in \sigma(\widetilde{D})$,
 then $C^*(G_{n,D})$ is not stably finite for every $n\ge 1$. 
\item\label{Heis_item3}  If $d_\zg\ne 0$ and $n\in 2\ZZ$, then  $C^*(G_{n,D})$ is not stably finite. 
\item\label{Heis_item4} If $d_\zg\ne 0$,  $n\in 2 \ZZ +1$, and there are $z_1, z_2\in \sigma(\widetilde{D})$ with 
$\Re\, z_1\le 0\le\Re\, z_2$, then $C^*(G_{n,D})$ is  AF-embeddable. 
\end{enumerate}
\end{proposition}

\begin{proof}
Assertion \eqref{Heis_item1} follows from Corollary~\ref{NC_cor1},  \eqref{Heis_item1.5} from Lemma~\ref{extra}, 
while \eqref{Heis_item2} is a consequence of Corollary~\ref{cf-cor8}\eqref{cf-cor7-i} 
and Example~\ref{ax+b}.
Assertion \eqref{Heis_item3} follows from \eqref{Heis_item1.5} along with  Corollary~\ref{solv-4n+2}, while \eqref{Heis_item4} can be obtained using Corollary~\ref{cf-cor8}\eqref{cf-cor7-ii} and Example~\ref{ax+b}.
\end{proof}

\subsection{Two more classes of examples}

We start with a lemma  that is essentially a by-product of \cite{Sp88}.
We prove it here for completeness, as we do not have a reference for this very result, and it is needed for Example~\ref{N6N15}, via
Lemma~\ref{special}. %{spatial}. 
We denote by $\mathcal{N}$ the class of separable nuclear $C^*$-algebras to which the universal coefficient theorem applies. 
(See for instance \cite{RoSc87}.)
%\cite[V.1.5.4]{Bl06}. 
The class~$\mathcal{N}$ contains the $C^*$-algebras of all simply connected solvable Lie groups, since they are obtained by iterated crossed products by actions of the group~$\RR$, starting from the 1-dimensional $C^*$-algebra. 

\begin{lemma}\label{embed}
Let $0\to\Ic\to\Ac\to\Ac/\Ic\to0$ be an exact sequence of $C^*$-algebras satisfying the following conditions: 
\begin{enumerate}[{\rm(i)}]
	\item\label{embed_item1} The $C^*$-algebra $\Ac/\Ic$ belongs to the class $\mathcal{N}$. 
	\item\label{embed_item2} The $C^*$-algebras $\Ic$ and $\Ac/\Ic$ are AF-embeddable. 
	\item\label{embed_item3} The index map $\delta_1\colon K_1(\Ac/\Ic)\to K_0(\Ic)$ vanishes. 
\end{enumerate}
Then the $C^*$-algebra $\Ac$ is AF-embeddable. 
\end{lemma}

\begin{proof}
Step 1 (reducing to essential ideals): By \cite[Lemma 1.12]{Sp88} and its proof we obtain a  $C^*$-algebra $\Ac'$ that fits in a commutative diagram 
$$\xymatrix{
\Ic \ar[r] \ar@{^{(}->}[d] & \Ac\ar[r] \ar@{^{(}->}[d] & \Ac/\Ic \ar[d] \\
\Ic\otimes\Kc \ar[r] & \Ac'\ar[r] & \Ac/\Ic
}
$$
where $\Ic\otimes\Kc$ embeds as an essential ideal of $\Ac'$, 
the first two vertical arrows give rise to group isomorphisms $K_*(\Ic)\simeq K_*(\Ic\otimes\Kc)$ ($\simeq K_*(\Ic)$) and $K_*(\Ac)\simeq K_*(\Ac')$, while the right-most vertical arrow is an automorphism of $\Ac/\Ic$. 
It then follows by the hypothesis~\eqref{embed_item3} along with the naturality of the index map (cf. \cite[Prop. 9.1.5]{RLL00}) 
that the index map $\delta_1\colon K_1(\Ac/\Ic)\to K_0(\Ic\otimes\Kc)$  of the bottom horizontal line in the above diagram vanishes. 

Step 2 (reducing to AF ideals): 
Since $\Ic$ is AF-embeddable, it follows that $\Ic\otimes\Kc$ is AF-embeddable, hence there exists an embedding $\Ic\otimes\Kc\hookrightarrow\widetilde{\Jc}$, where $\widetilde{\Jc}$ is an AF-algebra. 
Let $\Jc$ be the hereditary sub-$C^*$-algebra of $\widetilde{\Jc}$ generated by $\Ic\otimes\Kc$. 
Since $\widetilde{\Jc}$ is an AF-algebra, it then follows by \cite[Th. 3.1]{Ell76} that $\Jc$ is an AF-algebra. 
On the other hand $\Jc=\{bcb\mid 0\le b\in\Ic\otimes\Kc,\ c\in\widetilde{\Jc}\}$ by \cite[Cor. II.5.3.9]{Bl06}, 
which directly implies that every approximate unit of $\Ic\otimes\Kc$ 
is an approximate unit for~$\Jc$, too. 
That is, the embedding $\Ic\otimes\Kc\hookrightarrow\Jc$ is approximately unital in the sense of \cite[Def. 1.10]{Sp88}. 
Therefore we may use \cite[Rem. 1.11]{Sp88} 
to obtain a commutative diagram 
$$\xymatrix{
	\Ic\otimes\Kc \ar[r] \ar@{^{(}->}[d] & \Ac'\ar[r] \ar@{^{(}->}[d] & \Ac/\Ic \ar[d] \\
	\Jc \ar[r] & \Ac'+\Jc\ar[r] & \Ac/\Ic
}
$$
where the right-most vertical arrow is an automorphism of $\Ac/\Ic$. 
Since the index map $\delta_1\colon K_1(\Ac/\Ic)\to K_0(\Ic\otimes\Kc)$ 
of the upper line vanishes by Step~1, 
it then follows by the naturality of the index map 
that the index map $\delta_1\colon K_1(\Ac/\Ic)\to K_0(\Jc)$  of the bottom horizontal line in the above diagram vanishes as well. 
Now, since we have seen above that $\Jc$ is an AF-algebra, it follows by \cite[Lemma 1.13]{Sp88} applied to the short exact sequence 
$0\to\Jc \to  \Ac'+\Jc\to  \Ac/\Ic\to0$ 
that $\Ac'+\Jc$ is AF-embeddable. 
(The $C^*$-algebra $\Ac'+\Jc$ belongs to the class~$\mathcal{N}$ by the two-out-of-three property of that class mentioned in \cite[V.1.5.4]{Bl06}, so all the hypotheses of \cite[Lemma 1.13]{Sp88} are satisfied as stated.)

Step 3: The above Steps 1--2 give the embeddings $\Ac\hookrightarrow\Ac'\hookrightarrow\Ac'+\Jc$ 
along with the fact that $\Ac'+\Jc$ is AF-embeddable, 
hence $\Ac$ is AF-embeddable as well. 
\end{proof}

\begin{lemma}\label{free_L1}
Let $\ng$ be a nilpotent Lie algebra with its centre $\zg$, 
and assume that $D\in\Der(\ng)$. 
If there exists $\xi\in\ng^*\setminus\zg^\perp$ 
with $\xi\circ(\exp D)\in \Oc_\xi$, then  
$\sigma(D\vert_\zg)\cap 2\pi\ie \ZZ\ne\emptyset$. 
\end{lemma}

\begin{proof}
By hypothesis, there exists $X\in\gg$ with  $\xi\circ(\exp D)=\xi\circ\exp(\ad_\gg X)\in\gg^*$. 
This implies 
$\xi\circ(\exp D)\vert_{\zg}=\xi\circ\exp(\ad_\gg X)\vert_\zg\in\zg^*$. 
For every $Y\in\zg$ we have $\exp(\ad_\gg X)Y=Y$ 
and on the other hand,  since  $D$ is a derivation, $D(\zg)\subseteq\zg$. 
Therefore $\xi\circ \exp(D\vert_\zg)=\xi\vert_\zg$. 
Since $\xi\in\ng^*\setminus\zg^\perp$, that is, $\xi\vert_\zg\ne0$, 
we then obtain $1\in\spec(\exp(D\vert_\zg))$. 
Therefore, by the spectral mapping theorem, there exists $w\in\sigma(D\vert_\zg)$ with $\exp w=1$, that is, $w\in2\pi\ie \ZZ$. 
\end{proof}

\begin{lemma}\label{free_L2}
Let  $\ng$ be a nilpotent Lie algebra with its centre $\zg$, 
and $X:=(\ng^*\setminus\zg^\perp)/N$ endowed with its quotient topology. 
Then for arbitrary $D\in\Der(\ng)$ the map
$$\alpha\colon X\times \RR\to X,\quad 
(\Oc_\xi,t)\mapsto \alpha_t(\Oc_\xi):=(\exp(tD))^*(\Oc_\xi)=\Oc_{\xi\circ\exp(tD)}.$$
is well defined and is a continuous group action. 
Moreover, 
\begin{enumerate}[\rm (i)]
\item\label{free_L2_i} if
$\sigma(D\vert_\zg)\cap \ie \RR=\emptyset$, 
then the group action $\alpha$ is free; 
\item\label{free_L2_ii} if $X$ is Hausdorff and 
 there exists $\epsilon \in \{-1, 1\}$ such that 
$\epsilon \Re\, z> 0$ for every $z\in \sigma(D\vert_\zg)$,  
 the 
 action 
$\alpha$  is proper.
\end{enumerate}
\end{lemma}

\begin{proof}
	In order to check the equality 
	\begin{equation}\label{free_L2_proof_eq1}
	(\exp(tD))^*(\Oc_\xi)=\Oc_{\xi\circ\exp(tD)}
	\end{equation}
we use that for every $Y\in \ng$ and every $\gamma\in\Aut(\ng)$ one has $\gamma\circ \exp(\ad_\ng Y)\circ\gamma^{-1}=\exp(\ad_\ng \gamma(Y))$. 
Therefore, for $\gamma:=\exp(tD)^{-1}$, 
$$\xi\circ\exp(\ad_\ng Y)\circ \gamma^{-1}
=\xi\circ\gamma^{-1}\circ \exp(\ad_\ng \gamma(Y))\in\Oc_{\xi\circ\gamma^{-1}}.$$
Since the mapping $\gamma\colon\ng\to\ng$ is bijective, 
we then directly obtain~\eqref{free_L2_proof_eq1}. 

It is clear that $\alpha$ is a group action of the abelian group $(\RR,+)$, and its continuity follows by the commutative diagram 
$$\xymatrix{
(\ng^*\setminus\zg^\perp)\times\RR \ar[d]_{q\times\id_{\RR}} \ar[r]& \ng^*\setminus\zg^\perp \ar[d]^{q} \\
X \times\RR \ar[r]^{\alpha} & X
}$$
where $q\colon \ng^*\setminus\zg^\perp\to X$, $q(\xi):=\Oc_\xi$ is the quotient map defined by the coadjoint action of $N$, 
while the upper horizontal arrow is defined by $(\xi,t)\mapsto\xi\circ\exp(tD)$ and is clearly continuous.

\eqref{free_L2_i}  Assume that the group action $\alpha$ is not free, 
that is, there exist $t\in\RR^\times$ and $\xi\in \ng^*\setminus\zg^\perp$ 
with $\alpha_t(\Oc_\xi)=\Oc_\xi$. 
By \eqref{free_L2_proof_eq1}, we then have $\Oc_{\xi\circ\exp(tD)}=\Oc_\xi$, 
that is, $\xi\circ\exp(tD)\in\Oc_\xi$. 
Then Lemma~\ref{free_L1}  shows that  $\sigma(tD\vert_\zg)\cap 2\pi\ie \ZZ\ne\emptyset$, in particular $\sigma(D\vert_\zg)\cap \ie \RR\ne\emptyset$.  

\eqref{free_L2_ii} 
$X$ is a locally compact Hausdorff space.

Without losing the generality we assume that $\Re\, z > 0$ for every $z\in \sigma(D\vert_\zg)$. 
Let $\xi,\eta\in  \ng^*\setminus \zg^{\perp}$, $ (\xi_j)_{j\ge 1}$ be a sequence in 
 $\ng^*\setminus \zg^\perp $, and $(t_j)_{j\ge1}$ be a sequence in~$\RR$ such that 
 $\Oc_{\xi_j} \to \Oc_\xi $ 
and    $\alpha_{t_j}(\Oc_{\xi_j})\to \Oc_\eta $  in~$X$. 

Assume that $(t_j)_{j\ge 1}$ has no  limit points, hence it is not bounded. 
  It follows that there is a subsequence $(t_{j_k})_{k\ge 1}$ such that $t_{j_k} \to +\infty$ or $t_{j_k} \to -\infty$.
  Since $\Oc_{\xi_j} \to \Oc_\xi $ in $X$, there is $\xi'_j \in  \Oc_{\xi_j} $ such that 
  $\xi'_j \to \xi$, thus $\xi'_j \vert_{\zg}= \xi_j\vert_{\zg} \to \xi\vert_{\zg}$. 
  Similarly, since $\alpha_{t_j}(\Oc_{\xi_j})= \Oc_{\xi_j \circ \ee^{t_j D}}\to \Oc_\eta $ in  $X$, 
  it follows that $\xi_j \circ \ee^{t_j D}\vert_{\zg} \to \eta\vert_{\zg}$. 
  Assume that $t_{j_k}\to -\infty$. 
  Then  $\xi_{j_k} \circ \ee^{t_{j_k} D}\vert_{\zg} \to  0$ and we get that $\eta\vert_{\zg}=0$. 
  This is not possible since $\eta \in \ng^* \setminus \zg^\perp$. 
  If  $t_{j_k}\to +\infty$,
then $\xi_{j_k} \circ \ee^{t_{j_k} D}\vert_{\zg}\to +\infty$. 
This is again impossible since $\xi_{j_k} \circ \ee^{t_{j_k} D}\vert_{\zg}\to \eta\vert_\zg$. 
Therefore the sequence $(t_j)_{j\ge1}$ must have a limit point, hence the action $\alpha$ is proper. 
(See \cite[Lemma~3.42]{Wi07}.)
\end{proof}

For Lemma~\ref{special} below, we recall that if $\Vc$ is a finite-dimensional real vector space and $T\in\End(\Vc)$, then $T$ is called diagonalizable if the complexified space $\Vc_\CC:=\CC\otimes_\RR\Vc$ has a basis consisting of eigenvectors of $T_\CC:=\id_{\CC}\otimes T\in\End_\CC(\Vc_\CC)$, 
that is, one has the direct sum decomposition $\Vc_\CC=\bigoplus\limits_{\lambda\in\sigma(T)}\Vc_\CC(\lambda)$, 
where $\Vc_\CC(\lambda):=\Ker(T_\CC-\lambda\id_{\Vc_\CC})$ 
for any $\lambda\in\sigma(T):=\sigma(T_\CC)$. 
If this is the case and $\Yc\subseteq\Vc$ is an invariant subspace for $T$, then it is straightforward to check that 
\begin{equation}\label{reducing}
\Yc_\CC=\bigoplus\limits_{\lambda\in\sigma(T)}(\Yc_\CC\cap\Vc_\CC(\lambda))
\end{equation} 
and thus $T\vert_\Yc\in\End(\Yc)$ is diagonalizable. 
Moreover, selecting any $\CC$-linear subspace $\Xc_\lambda\subseteq \Vc_\CC(\lambda)$ with $\Vc_\CC(\lambda)=\Xc_\lambda\dotplus(\Yc_\CC\cap\Vc_\CC(\lambda))$ for every $\lambda \in\sigma(T)$, and defining 
$\Xc:=\bigoplus\limits_{\lambda\in\sigma(T)}\Xc_\lambda$, 
we obtain a $\CC$-linear subspace $\Xc\subseteq\Vc_\CC$ 
that is invariant to $T_\CC$ and satisfies $\Vc_\CC=\Xc\dotplus\Yc_\CC$. 
Moreover, defining $\widetilde{T}\in\End(\Vc/\Yc)$, $v+\Yc\mapsto Tv+\Yc$, 
we have $\CC$-linear isomorphisms 
$(\Vc/\Yc)_\CC\simeq \Vc_\CC/\Yc_\CC\simeq \Xc$ that intertwine the operators 
$\widetilde{T}_\CC\in\End_\CC((\Vc/\Yc)_\CC)$ and $T_\CC\vert_\Xc\in\End_\CC(\Xc)$, 
hence
\begin{equation}\label{quot}
\sigma(\widetilde{T})=\sigma(\widetilde{T}_\CC)=\sigma(T_\CC\vert_\Xc).
\end{equation}

\begin{lemma}\label{special}
Let $H$ be a nilpotent Lie group with its Lie algebra $\hg$ and $\zg$ the centre of $\hg$.
Assume the following conditions hold:
\begin{enumerate}[{\rm (i)}]
\item\label{cond_i}  %$\hg$ is two-step nilpotent and 
$[\hg, \hg]=\zg$ and $\dim(\hg/\zg)\in2\ZZ+1$.
\item\label{cond_ii} The  non-trivial coadjoint orbits of $H$ have the same dimension $d$, that is, there is 
$d\in 2\ZZ$ such that 
$$ \hg^*/H= (\hg^*/H)_d \sqcup [\hg, \hg]^{\perp}.$$
Here $(\hg^*/H)_d$ denotes the space of coajoint orbits of $H$ of dimension $d$. 
\end{enumerate}
If $D\in \Der(\hg)$ is diagonalizable, satisfies $\sigma(D)\cap\ie\RR\subseteq\{0\}$, and there exist $z_1, z_2 \in \sigma(D)$  with $(\Re\, z_1) (\Re\, z_2) \le 0$, 
then $C^*(H \rtimes_D \RR)$ is AF-embeddable. 
\end{lemma}

\begin{proof}
Assume first that there are $z_1, z_2 \in \sigma(D\vert _{\zg})$  such that $(\Re z_1) (\Re z_2) \le 0$. 
Then $C^*(H\rtimes_D \RR)$ is AF-embeddable, by Proposition~\ref{NC}, and this proves the lemma in this case.

Next assume that $\epsilon\in \{-1, 1\}$ such that  
\begin{equation}\label{roots_1}
\sigma(D \vert_{\zg}) \subseteq \epsilon (0, \infty) +\ie \RR. 
\end{equation}
Define $\widetilde{D}\in\Der(\hg /\zg)$, $X+\zg\mapsto D(X)+\zg$. 
We claim that 
\begin{equation}\label{add}
\sigma(D\vert_\zg)\subseteq\{z+w\mid z,w\in\sigma(\widetilde{D})\}.
\end{equation}
For proving \eqref{add}, we use the remarks preceding the statement of the present lemma with $\Vc:=\hg$, $\Yc:=\zg$, and $T:=D$.  
We obtain $\hg_\CC=\Xc\dotplus\zg_\CC$ hence the hypothesis 
$\zg=[\hg,\hg]=\zg$ implies $\zg_\CC=[\hg_\CC,\hg_\CC]=[\Xc,\Xc]$. 
Since $D_\CC\in\End_\CC(\hg_\CC)$ is a derivation of the complex Lie algebra $\hg_\CC$, for any $\lambda_1,\lambda_2\in\sigma(D)$ we have $[\hg_\CC(\lambda_1),\hg_\CC(\lambda_2)]\subseteq\hg_\CC(\lambda_1+\lambda_2)$. 
Since $\Xc=\bigoplus\limits_{\lambda\in\sigma(D_\CC\vert_\Xc)}\Xc_\lambda$
and $\zg_\CC=[\Xc,\Xc]$, we then obtain 
$$\zg_\CC\subseteq\bigoplus_{\lambda_1,\lambda_1\in \sigma(D_\CC\vert_\Xc)}
\hg_\CC(\lambda_1+\lambda_2)$$
and then, by~\eqref{reducing}, 
$$\zg_\CC=\bigoplus_{\lambda_1,\lambda_1\in \sigma(D_\CC\vert_\Xc)}
\zg_\CC\cap\hg_\CC(\lambda_1+\lambda_2).$$
Now, since $\sigma(\widetilde{D})=\sigma(D_\CC\vert_\Xc)$ by \eqref{quot}, 
we directly obtain~\eqref{add}.

Then~\eqref{roots_1} and \eqref{add} imply that there exist
\begin{equation}\label{roots_2}
z_1, z_2 \in \sigma(\widetilde{D}) \quad \text{such that} \; \; \Re\, z_1\le 0\le \Re\, z_2.
\end{equation}
Then there is a short exact sequence
$$ 0 \rightarrow \Ic \rightarrow  C^*(H)\rightarrow C^*(H/Z)\rightarrow 0, 
$$
where $\widehat{\Ic} = (\hg^*/H)_d$ and $Z=\exp_H \zg$ is the centre of $H$.
Since $\hg$ is two step nilpotent,  hence  it has only flat orbits, it follows from \cite[Lemma~6.8]{BBL17} that $\Ic$ has continuous trace 
hence $\widehat{\Ic}$ is Hausdorff. 
Since $ C^*(H/Z)$ is invariant under automorphisms of $H$, $\Ic$ is $\Aut(H)$ invariant as well. 
We  thus obtain the short exact sequence
\begin{equation}\label{ses}
0 \rightarrow \Ic\rtimes  \RR  \rightarrow  C^*(H\rtimes_{D} \RR) \rightarrow C^*((H/Z)\rtimes_{\widetilde{D}} \RR )\rightarrow 0.  
\end{equation}
Here, we have that  $\sigma(D\vert_\zg) \cap \ie \RR =\emptyset$ 
by \eqref{roots_1},  hence,  by Lemma~\ref{free_L2}, the  continuous action 
$\alpha_{D}  \colon \RR \times \widehat{\Ic} \to \widehat{\Ic}$ is free and proper. 
Thus $\Ic \rtimes_{D}  \RR$  has continuous trace by \cite[Th. 1.1(3)]{RaRo88}, hence it is AF-embeddable by \cite[Lemma 2.6]{Sp88}. 
It then follows from condition \eqref{roots_2}  and  Example~\ref{ax+b}  that 
$C^*((H/Z)\rtimes_{\widetilde{D}} \RR)$ is AF-embeddable as well, since the group $H/Z$ is abelian. 

Since $\dim(\hg/\zg)\in2\ZZ+1$, then $K_1(C^*((H/Z)\rtimes_{\widetilde{D}} \RR) ) = K_0 (C^*(H/Z))=\{0\}$, hence the index 
map corresponding to the $C^*$-algebra extension~\eqref{ses} vanishes. 
Therefore, by Lemma~\ref{embed}, $C^*(H\rtimes_{D} \RR)$ is AF-embeddable. 
\end{proof}

\begin{remark}
\normalfont
In Lemma~\ref{special}, if the hypothesis $\dim(\hg/ \zg)\in2\ZZ+1$ is dropped, then the index map corresponding to \eqref{ses} may not vanish, 
and $C^*(H\rtimes_{D} \RR)$ may not be AF-embeddable, or even stably finite, as we will see in Theorem~\ref{N6N17} below.
\end{remark}

\begin{remark}\label{deriv-ext}
\normalfont
Let $\hg$ be a finite-dimensional real Lie algebra with a symplectic structure $\omega\colon\hg\times\hg\to\RR$, and define the corresponding central extension 
$\ng:=\hg\dotplus_\omega\RR$ with its Lie bracket $[(X_1,t_1),(X_2,t_2)]:=([X_1,X_2],\omega(X_1,X_2))$ for all $X_1,X_2\in\hg$ and $t_1,t_2\in\RR$. 
For any $D_0\in\Der(\hg)$ and $a_0\in\RR$ we define the linear map $D\colon\ng\to\ng$, $D(X,t):=(D_0X,a_0t)$. 
Then  $D\in\Der(\ng)$ if and only if 
$$\omega(D_0X_1,X_2)+\omega(X_1,D_0X_2)=a_0\omega(X_1,X_2) 
\text{ for all }X_1,X_2\in\hg.$$
\end{remark}

\begin{lemma}\label{N6N15-lemma}
	Let $\hg$ be the real Lie algebra with a basis $X_1,X_2,X_3,Y_1,Y_2,Y_3$ satisfying the commutation relations $[X_1,X_2]=Y_3$, $[X_2,X_3]=Y_1$,  $[X_3,X_1]=Y_2$. 
	\begin{enumerate}[{\rm(i)}]
		\item\label{N6N15_item0}  The centre of $\hg$ is $\zg:=\spa\{Y_1,Y_2,Y_3\}$ and  $\dim\Oc_\xi=2$ for all $\xi\in\hg^*\setminus\zg^\perp$. 
		\item\label{N6N15_item1} 
		For any  $a_1,a_2,a_3\in\RR$ there exists a unique skew-symmetric bilinear functional 
		$\omega\colon\hg\times\hg\to\RR$ 
		with $\omega(X_j,Y_k)=\delta_{jk}a_j$, $\omega(X_j,X_k)=\omega(Y_j,Y_k)=0$ for all $j,k\in\{1,2,3\}$. 
		Moreover, $\omega$ is a symplectic structure of the Lie algebra $\hg$ if and only if 
		\begin{equation}\label{N6N15_item1_eq1}
		a_1+a_2+a_3=0\text{ and }a_1a_2a_3\ne0.
		\end{equation}
		\item\label{N6N15_item2} 
		For any matrix $B=(b_{jk})_{1\le j,k\le3}\in M_3(\RR)$ there exists a unique derivation $D_B\in\Der(\hg)$ satisfying $D_BX_j=\sum\limits_{k=1}^3b_{jk}X_k$ for $j=1,2,3$. 
		If $a_1,a_2,a_3\in\RR$ satisfy~\eqref{N6N15_item1_eq1} and $\omega$ is their corresponding symplectic structure of~$\hg$ as in~\eqref{N6N15_item1} above, then there exists $D\in\Der(\hg\dotplus_\omega\RR)$ with $D\vert_\hg=D_B$ 
		if and only if 
		\begin{equation}\label{N6N15_item2_eq2}
		b_{ij}(a_i-a_j)=0\text{ for all }j,k\in\{1,2,3\}
		\end{equation}
		and if this is the case then $D(0,1)=(0,\Tr B)\in\hg\dotplus_\omega\RR$. 
		\end{enumerate}
\end{lemma}

\begin{proof}
\eqref{N6N15_item0} This is well known.  

\eqref{N6N15_item1} 
It is straightforward to check that $\omega$ is a 2-cocycle if and only if $a_1+a_2+a_3=0$, and on the other hand $\omega$ is non-degenerate if and only if $a_1a_2a_3\ne0$. 
Therefore, $\omega$ is a symplectic structure of the Lie algebra $\hg$ if and only if \eqref{N6N15_item1_eq1} is satisfied. 

\eqref{N6N15_item2} 
In order to obtain a derivation $D_B\colon\hg\to\hg$ we define $D_BY_j:=[D_B X_r,X_s]+[X_r,D_BX_s]$ if $Y_j=[X_r,X_s]$. 
A straightforward computation then leads to the formula 
\begin{equation}\label{N6N15_proof_eq1}
D_BY_j=-\sum_{k\ne j}b_{kj}Y_k+\Bigl(\sum_{k\ne j}b_{kk})Y_j
\text{ for }j=1,2,3,
\end{equation}
which further implies 
$$\omega(D_B X_i,Y_j)+\omega(X_i,D_B Y_j)=
\begin{cases}
(\Tr B)a_j=(\Tr B)\omega(X_j,Y_j)&\text{ if }i=j,\\
b_{ij}(a_i-a_j)&\text{ if }i\ne j.
\end{cases}
$$
Since $\omega(X_i,Y_j)=0$ if $i\ne j$, the assertion then follows by Remark~\ref{deriv-ext}. 
\end{proof}

\begin{theorem}\label{N6N15}
Let $\hg$ be the real Lie algebra with a basis $X_1,X_2,X_3,Y_1,Y_2,Y_3$ satisfying the commutation relations $$[X_1,X_2]=Y_3,\ [X_2,X_3]=Y_1,\ [X_3,X_1]=Y_2. $$
For any  $a_1,a_2,a_3\in\RR$ with $a_1+a_2+a_3=0\ne a_1a_2a_3$, 
let $\omega\colon\hg\times\hg\to\RR$ be the symplectic form 
satisfying $\omega(X_j,Y_k)=\delta_{jk}a_j$, $\omega(X_j,X_k)=\omega(Y_j,Y_k)=0$ for all $j,k\in\{1,2,3\}$, 
and define the corresponding central extension 
$$\ng:=\hg\dotplus_\omega\RR.$$ 
For any $b_1, b_2, b_3 \in \RR$ with $b_1+b_2+b_3\ne 0$, let $D\in \Der(\ng)$ be the unique derivation with 
		$D X_j= b_{j}X_j$ for $j=1,2,3$.
		If we denote by $N\rtimes_D\RR$ the 8-dimensional exponential solvable Lie group whose Lie algebra is $\ng\rtimes\RR D$, 
		then the following assertions are equivalent: 
		\begin{itemize}
			\item $C^*(N\rtimes_D\RR)$ is not AF-embeddable. 
			\item $C^*(N\rtimes_D\RR)$ is not stably finite. 
			\item There exists $\epsilon\in\{\pm1\}$ with 
			$\epsilon z> 0$ for every $z\in \sigma(D)$. 
			\item There exists $\epsilon\in\{\pm1\}$ with 
			$\epsilon b_j > 0$ for $j=1,2,3$. 
		\end{itemize}
\end{theorem}

\begin{proof} 
The existence and uniqueness of the symplectic form $\omega$ and the derivation~$D$ as in the statement follow by Lemma~\ref{N6N15-lemma}. 

The last two assertions in the statement are clearly equivalent since 
$$\sigma(D)=\{b_j\mid j=1,2,3\}\cup\Bigl\{\sum_{k\ne j}b_k\mid j=1,2,3\Bigr\}
\cup\{b_1+b_2+b_3\}$$
by Lemma~\ref{N6N15-lemma}\eqref{N6N15_item2}  and \eqref{N6N15_proof_eq1}. 
This also shows that $\sigma(D)\subseteq\RR$, hence $\ng\rtimes \RR D$ is 
an exponential (actually, a completely solvable) Lie algebra. 

If there exists $\epsilon\in\{\pm1\}$ with 
$\epsilon z> 0$ for every $z\in \sigma(D)$, then $C^*(N\rtimes_D\RR)$ is not stably finite by 
%Corollary~\ref{cor-cf5}. 
Theorem~\ref{prop-cf4}. 

Conversely, assume that there exists no $\epsilon\in\{\pm1\}$ with 
$\epsilon b_j > 0$ for $j=1,2,3$.  
Then, denoting $D_B:=D\vert_\hg\in\Der(\hg)$ as in Lemma~\ref{N6N15-lemma}, 
we have 
$$\sigma(D_B)=\{b_j\mid j=1,2,3\}\cup\Bigl\{\sum_{k\ne j}b_k\mid j=1,2,3\Bigr\}$$
hence $\sigma(D_B)\cap\ie\RR\subseteq\{0\}$ and  
there exist $z_1, z_2 \in \sigma(D_B)$  with $(\Re\, z_1) (\Re\, z_2) \le 0$. 
Therefore  $C^*(H\rtimes_D\RR)$ is AF-embeddable by 
Lemma~\ref{special} along with Lemma~\ref{N6N15-lemma}\eqref{N6N15_item0}. 
Now, since $\dim(N\rtimes_D\RR)=8\in 4\ZZ$, 
Corollary~\ref{cf-cor8}\eqref{cf-cor7-ii} implies that 
$C^*(N\rtimes_D\RR)$ is AF-embeddable, 
and this finishes the proof. 
\end{proof}

For Lemma~\ref{N6N17-lemma} and Theorem~\ref{N6N17} below, we recall that a \emph{finite-dimensional real division algebra} is a finite-dimensional real vector space $\KK$ endowed with a bilinear map $\KK\times\KK\to\KK$, $(v,w)\mapsto v w$, 
whose corresponding linear mappings $v\mapsto v w_0$ and $w\mapsto v_0 w$ are injective (hence bijective) for all $v_0,w_0\in\KK\setminus\{0\}$. 
If this is the case, then $\dim_\RR\KK\in\{1,2,4,8\}$ by \cite[Cor. 1]{BoMi58}, 
and these values of $\dim_\RR\KK$ are realized for instance if $\KK$ is the real field~$\RR$, the complex field~$\CC$, the quaternion field~$\mathbb{H}$, and the octonion (non-associative) algebra~$\mathbb{O}$, respectively. 

Let $\KK$ be a finite-dimensional real division algebra and define 
$$
\begin{gathered}
\omega\colon(\KK^n\times\KK^n)\times(\KK^n\times\KK^n)\to\KK,\\
\omega((v_1,w_1),(v_2,w_2)):=\sum_{k=1}^n(v_{1k}w_{2k}-v_{2k}w_{1k})
\end{gathered}
$$
for $v_j=(v_{j1},\dots,v_{jn}), w_j=(w_{j1},\dots,w_{jn})\in\KK^n$, $j=1,2$. 
It is clear that $\omega((v_1,w_1),(v_2,w_2))=-\omega((v_2,w_2),(v_1,w_1))$, 
hence we may define the \emph{real} 2-step nilpotent Lie algebra 
\begin{equation}\label{hgK}
\hg_\KK:=\KK^n\times\KK^n\times\KK
\end{equation}
with its Lie bracket 
$[(v_1,w_1,z_1),(v_2,w_2,z_2)]:=[(0,0,\omega((v_1,w_1),(v_2,w_2)))]$. 
Let $H_\KK=(\hg_\KK,\cdot)$ be the 
%connected, simply connected 
nilpotent Lie group whose Lie algebra is~$\hg_\KK$.

\begin{lemma}\label{N6N17-lemma}
Let $\KK$ be a finite-dimensional real division algebra, and define 
  $\zg:=\{0\}\times\{0\}\times\KK\subseteq\hg_\KK$. Then the following assertions hold: 
\begin{enumerate}[{\rm(i)}]
	\item\label{N6N17_item1} We have $[\hg_\KK,\hg_\KK]=\zg$ and $\zg$ is the centre of $\hg$. 
	\item\label{N6N17_item2} For every $\xi\in\hg_\KK^*\setminus[\hg_\KK,\hg_\KK]^\perp$ we have $\hg_\KK(\xi)=\zg$ and $\dim\Oc_\xi=2n\dim_\RR\KK$. 
	\item\label{N6N17_item3} 
	The mapping
	$r_\zg\colon (\hg_\KK^*\setminus\zg^\perp)/H_\KK\to \zg^*\setminus\{0\}
	,\quad \Oc_\xi\mapsto\xi\vert_{\zg}$
	is a well-defined homeomorphism. 
	\item\label{N6N17_item4} If $a=(a_1,\dots,a_n),b=(b_1,\dots,b_n)\in\RR^n$, $c\in\RR$, and $D\colon\hg_\KK\to\hg_\KK$ is the $\RR$-linear mapping defined by 
	$$D(v,w,z):=((a_1v_1,\dots,a_nv_n),(b_1w_1,\dots,b_nw_n),cz), $$ 
	then $D\in\Der(\hg_\KK)$ if and only if $a_k+b_k=c$ for $k=1,\dots,n$. 
	\end{enumerate}
 \end{lemma}

\begin{proof}
 \eqref{N6N17_item1} 
It is clear that $[\hg_\KK,\hg_\KK]=\zg$ and $[\hg_\KK,\zg]=\{0\}$.
In order to prove that $\zg$ is actually equal to the centre of $\hg_\KK$, 
let us assume that there exists $x_0:=(v_0,w_0,z)\in\hg_\KK\setminus\zg$ with $[x_0,\hg_\KK]=\{0\}$ and $v_0=(v_{01},\dots,v_{0n}), w_0=(w_{01},\dots,w_{0n})\in\KK^n$. 
Since $x_0\not\in\zg$, there exists $j\in\{1,\dots,n\}$ with $v_{0j}\in\KK\setminus\{0\}$ or $w_{0j}\in\KK\setminus\{0\}$. 
If for instance $v_{0j}\ne0$, then we define $x:=(v,w,0)\in\hg$ 
where $v=0\in\KK^n$ and $w=(w_1,\dots,w_n)\in\KK^n$ 
is given by $w_k=0$ if $k\in\{1,\dots,n\}\setminus\{j\}$ and $w_j=v_{0j}\in\KK$, 
and we obtain $[x_0,x]=(0,0,v_{0j}v_{0j})\in\hg\setminus\{0\}$, which is a contradiction with the assumption $[x_0,\hg]=\{0\}$. 
The case $w_j\in\KK\setminus\{0\}$ can be discussed similarly. 

\eqref{N6N17_item2} 
We have $\hg_\KK(\xi)=\{x\in\hg\mid[x,\hg_\KK]\subseteq\Ker\xi\}$, so the inclusion $\zg\subseteq\hg_\KK(\xi)$ follows by~\eqref{N6N17_item1}. 
For the converse inclusion, assume there exists $x_0:=(v_0,w_0,z)\in\hg_\KK(\xi)\setminus\zg$. 
Since $x_0\not\in\zg$, there exists $j\in\{1,\dots,n\}$ with $v_{0j}\in\KK\setminus\{0\}$ or $w_{0j}\in\KK\setminus\{0\}$. 
If for instance $v_{0j}\ne0$, then for any $y=(v,w,0)\in\hg$ with $v=0\in\KK^n$ and $w=(w_1,\dots,w_n)\in\KK^n$ 
with $w_k=0$ if $k\in\{1,\dots,n\}\setminus\{j\}$ we have 
$[x_0,x]=(0,0,v_{0j}w_j)$. 
Here $w_j\in\KK$ is arbitrary and $v_{0j}\in\KK\setminus\{0\}$ hence, 
since $\KK$ is a division algebra, it follows that $\zg\subseteq[x_0,\hg_\KK]$. 
On the other hand, we have by assumption $x_0\in\hg_\KK(\xi)$, hence 
$[\hg_\KK,\hg_\KK]=\zg\subseteq[x,\hg_\KK]\subseteq\Ker\xi$, 
%which is 
a contradiction with the hypothesis $\xi\in\hg_\KK^*\setminus[\hg_\KK,\hg_\KK]^\perp$. 

The second assertion follows by the general equality $\dim\Oc_\xi=\dim(\hg_\KK/\hg_\KK(\xi))$. 

\eqref{N6N17_item3} 
For every  $\xi\in\hg_\KK^*\setminus[\hg_\KK,\hg_\KK]^\perp$ we have $\Oc_\xi=\xi+\zg^\perp$ by~\eqref{N6N17_item3} hence the mapping $r_\zg$ is well-defined and bijective. 
Moreover, if we define $r\colon \hg_\KK^*\setminus\zg^\perp\to \zg^*\setminus\{0\}$, $\xi\mapsto\xi\vert_\zg$, and $q\colon \hg_\KK^*\setminus\zg^\perp\to(\hg_\KK^*\setminus\zg^\perp)/H_\KK$, $\xi\mapsto \Oc_\xi$, 
then $r_\zg\circ q=r$ and, since $r$ is a continuous open mapping and $q$ is a quotient mapping, it follows that $r_\zg$ is continuous and open. 

\eqref{N6N17_item4} 
This assertion is straightforward. 
\end{proof}

\begin{theorem}\label{N6N17}
Let $\KK$ be a finite-dimensional real division algebra.
Assume $a=(a_1,\dots,a_n),b=(b_1,\dots,b_n)\in\RR^n$, $c\in\RR$,
are such that  $a_k+b_k=c \ne 0$ for $k=1,\dots,n$.
and let  $D\in \Der(\hg_\KK)$ be the $\RR$-linear mapping defined by 
	$D(v,w,z):=((a_1v_1,\dots,a_nv_n),(b_1w_1,\dots,b_nw_n),cz)$.
If  $H_\KK\rtimes_D\RR$ is the simply connected Lie group whose Lie algebra is $\hg_\KK\rtimes\RR D$, then 
	$\Pg(C^*(H_\KK\rtimes_D \RR))\ne\{0\}$. 
	If moreover $\KK\ne\RR$, then $C^*(H_\KK\rtimes_D \RR)$ is not stably finite. 
 \end{theorem}

\begin{proof} 
Since $H_\KK$ is a nilpotent Lie group, we may use Kirillov's homeomorphism 
$\widehat{H_\KK}\simeq\hg_\KK^*/H_\KK$ along with \eqref{N6N17_item3} to obtain a short exact sequence 
$$0\to\Ic\hookrightarrow C^*(H_\KK)\to C^*(H_\KK/Z)\to 0$$
where one has $*$-isomorphisms 
$\Ic\simeq \Cc_0(\zg^*\setminus\{0\})\otimes\Kc(\Hc_0)$ 
for some separable infinite-dimensional complex Hilbert space~$\Hc_0$
and $C^*(H_\KK/Z)\simeq \Cc_0(\zg^\perp)$. 
The derivation $D$ gives rise to a group action $\alpha\colon (\hg_\KK^*/H_\KK)\times\RR\to\hg_\KK^*/H_\KK$, $(\Oc_\xi,t)\mapsto \Oc_{\xi\circ\ee^{tD}}$. 
If we fix any real scalar product on $\zg$ and we denote by $S_{\zg^*}$ its corresponding unit sphere, then we have the homeomorphism  
$$S_{\zg^*}\times\RR\to\zg^*\setminus\{0\},\quad (\eta,t)\mapsto \ee^{tc}\eta$$
since $c\in\RR\setminus\{0\}$, 
which gives 
a $*$-isomorphism $\Cc_0(\zg^*\setminus\{0\})\simeq\Cc(S_{\zg^*})\otimes \Cc_0(\RR)$. 
We then obtain $*$-isomorphisms 
\begin{align*}
\Ic\rtimes_\alpha\RR
& \simeq (\Cc_0(\zg^*\setminus\{0\})\otimes\Kc(\Hc_0))\rtimes_\alpha\RR
\simeq \Cc(S_{\zg^*})\otimes (\Cc_0(\RR)\rtimes_c\RR)\otimes\Kc(\Hc_0) \\
& \simeq \Cc(S_{\zg^*})\otimes \Kc(L^2(\RR))\otimes\Kc(\Hc_0) 
 \simeq \Cc(S_{\zg^*})\otimes\Kc(\Hc_0).
\end{align*}
The above crossed product $\Cc_0(\RR)\rtimes_c\RR$ is defined via the group action 
$$\Cc_0(\RR)\times\RR\to \Cc_0(\RR), \quad (\varphi,t)\mapsto \varphi_t, \text{ where  }\varphi_t(s):=\varphi(s\ee^{tc}),$$ 
hence one has a $*$-isomorphism 
$\Cc_0(\RR)\rtimes_c\RR\simeq \Kc(L^2(\RR))$ since $c\in\RR\setminus\{0\}$. 
The above $*$-isomorphisms show that $\Pg(\Ic\rtimes_\alpha\RR)\ne\{0\}$. 
Since $\Ic\rtimes_\alpha\RR$ is an ideal of $C^*(H_\KK)\rtimes_\alpha\RR$, 
and $C^*(H_\KK)\rtimes_\alpha\RR\simeq C^*(H_\KK\rtimes_D\RR)$, we obtain 
$\Pg(C^*(H_\KK \rtimes_D \RR))\ne\{0\}$. 

If moreover $\KK\ne\RR$, then $\dim_\RR\KK\in 2\ZZ$, hence $\dim(H_\KK\rtimes_D\RR)\in 2\ZZ+1$, 
and it follows that $C^*(H_\KK\rtimes_D \RR)$ is not stably finite by Theorem~\ref{4n+2}. 
\end{proof}

\subsection*{Acknowledgment}
We wish to thank the Referee for several remarks that improved the presentation. 

The research of the second-named author was supported by
a grant of the Ministry of Research, Innovation and Digitization, CNCS/CCCDI –
UEFISCDI, project number PN-III-P4-ID-PCE-2020-0878, within PNCDI III.

\end{document}